\documentclass[12pt,leqno]{amsart}
\usepackage{a4,enumerate,amssymb}

\swapnumbers
\numberwithin{equation}{section}
\newtheorem{thm}[equation]{Theorem}
\newtheorem*{thm*}{Theorem}
\newtheorem{prop}[equation]{Proposition}
\newtheorem{cor}[equation]{Corollary}

\newtheorem*{conj*}{Conjecture}
\newtheorem{lem}[equation]{Lemma}

\theoremstyle{definition}

\newtheorem{ex}[equation]{Example}

\newtheorem{rem}[equation]{Remark}

\newcommand{\op}{{\mathsf{op}}}
\newcommand{\nat}{\mathbb{N}}
\newcommand{\id}{\mathsf{id}}
\newcommand{\s}{\sigma}
\newcommand{\matr}[1]{\mathbb{M}_{#1}}
\DeclareMathOperator{\car}{char}
\DeclareMathOperator{\Symm}{Sym}
\DeclareMathOperator{\Skew}{Skew}
\DeclareMathOperator{\Symd}{Symd}
\DeclareMathOperator{\Syms}{Sym^\ast}
\DeclareMathOperator{\kap}{cap}
\DeclareMathOperator{\Prp}{Prp}
\DeclareMathOperator{\Pc}{Pc}
\DeclareMathOperator{\Tr}{Tr}
\DeclareMathOperator{\Prd}{Prd}
\DeclareMathOperator{\ind}{ind}
\DeclareMathOperator{\coind}{coind}
\DeclareMathOperator{\Int}{Int}
\DeclareMathOperator{\End}{End}

\DeclareMathOperator{\sw}{\mathsf{sw}}

\DeclareMathOperator{\Cor}{Cor}

\DeclareMathOperator{\rad}{rad}
\newcommand{\mg}[1]{#1^{\times}}

\newcommand{\ovl}{\overline}
\renewcommand{\leq}{\leqslant}
\renewcommand{\geq}{\geqslant}
\newcommand{\la}{\langle}
\newcommand{\ra}{\rangle}

\title{Involutions and stable subalgebras}
\date{October 2017}
\author[K.J. Becher]{Karim Johannes Becher}
\author[N. Grenier-Boley]{Nicolas Grenier-Boley} 
\author[J.-P. Tignol]{Jean-Pierre Tignol}
\address{Universit\"at Konstanz, Zukunftskolleg/FB Mathematik und Statistik, D-78457 Konstanz, Germany.}
\address{Universiteit Antwerpen, Departement Wiskunde en Informatica, Middelheimlaan~1, B-2020 Antwerpen, Belgium}
\email{KarimJohannes.Becher@uantwerpen.be}
\address{Universit\'e Rouen Normandie, Laboratoire de
  Didactique Andr\'e Revuz\linebreak (LDAR, EA 4434), Universit\'e Artois, Universit\'e Cergy-Pontoise,  Universit\'e Paris Diderot, Universit\'e Paris Est Cr\'eteil, F-76130 Mont-Saint-Aignan, France.}
  \email{nicolas.grenier-boley@univ-rouen.fr}
\address{Universit\'e catholique de Louvain, ICTEAM Institute, Avenue
  G.~Lema\^{\i}tre 4, Box~L4.05.01,
B-1348 Louvain-La-Neuve, Belgium.}
\email{jean-pierre.tignol@uclouvain.be}

\begin{document}

\begin{abstract}

Given a central simple algebra with involution over an arbitrary field, \'etale subalgebras contained in the space of symmetric elements are investigated.
The method emphasizes the similarities between the various types of involutions and privileges a unified treatment for all characteristics whenever possible.
As a consequence a conceptual proof of a theorem of Rowen is obtained, which
asserts that every division algebra of exponent two and degree eight contains a maximal subfield that is a triquadratic extension of the centre.


\medskip\noindent
\sc{Classification (MSC 2010):} 16H05, 16R50, 16W10
\end{abstract}

\maketitle
\section{Introduction}

We investigate \'etale algebras in the space of symmetric elements of
a central simple algebra with involution over an arbitrary field, emphasizing the similarities between the various types of involutions and avoiding restrictions on the characteristic. 
In Section~\ref{S:algebras} and Section~\ref{S:cap} we recall the terminology and some crucial techniques for algebras with involution. 
We enhance this terminology in a way that allows us to avoid unnecessary case distinctions in the sequel, according to the different types of involution and to the characteristic.
To this end we introduce in Section~\ref{S:cap} the notion of \emph{capacity} of an algebra with involution. It is defined to be the degree of the algebra if the involution is orthogonal or unitary, and half the degree if the involution is symplectic.
In Section~\ref{S:neat} we isolate a notion of \emph{neat} subalgebra, which captures the features of separable field extensions of the centre  consisting of symmetric elements while avoiding the pathologies that may arise with arbitrary \'etale algebras. 
We prove their existence and determine their maximal dimension to be
equal to the capacity (Theorem~\ref{thm:capmaxdim} and
  Proposition~\ref{prop:maxneat}). 
In Section~\ref{S:neat-quad}, given a neat quadratic subalgebra $K$, we establish the existence of a neat subalgebra $L$ linearly disjoint from~$K$ and
centralising $K$ and such that the composite $KL$ is a neat algebra of maximal dimension (Theorem~\ref{thm:neatquad}). 
In Section~\ref{S:cap4} we apply this result to construct neat biquadratic subalgebras in the space of symmetric elements of central simple algebras of degree~$4$ with orthogonal or unitary involutions, and similarly of central simple algebras of degree $8$ with symplectic involutions (Theorem~\ref{T:neat-quad-sym-exist}). 
As a consequence, we obtain a conceptual proof of a theorem of Rowen, which asserts that division algebras of exponent~$2$ and degree~$8$ are
elementary abelian crossed products, i.e., they contain a maximal subfield which is a triquadratic separable extension of the centre (Corollary~\ref{C:AR}).
Actually we obtain directly a refined version of this result which says that any symplectic involution on a central simple algebra of degree $8$ stabilizes some triquadratic \'etale extension of the centre (Theorem~\ref{T:symp-deg8-triquad}).
This has been proven in \cite[Lemma~6.1]{GPT09} for division algebras in characteristic different from two, but there the proof uses Rowen's Theorem, which we obtain here as a consequence.
This illustrates the usefulness of involutions in the investigation of
central simple algebras of exponent two.

The results of this paper will be used in \cite{BGBT3}, which proposes
a common approach to the definition of the first cohomological
invariant (discriminant) of the involutions of capacity four of
various types through
Pfister forms in arbitrary characteristic.

\section{Algebras} 
\label{S:algebras}

In this preliminary section we introduce and recall some definitions and facts from the theory of finite-dimensional simple and semisimple algebras.
Our standard references are \cite{Pierce} and \cite{GS06}.
\medskip

Let $F$ be an arbitrary field.
For a commutative $F$-algebra $K$ we set $[K:F]=\dim_FK$. 
Recall that an $F$-algebra is \emph{\'etale} if it is isomorphic to a finite product of finite separable field extensions of $F$.
An \'etale $F$-algebra is said to
be \emph{split} if it is $F$-isomorphic to $F^n$ for some $n\in\nat$.

\begin{lem}\label{L:split-etale-infinite}
Let $L$ be a split \'etale $F$-algebra with $|F|>[L:F]$.
Then $L=F[a]$ for an element $a\in \mg{L}$ which is separable over $F$.
\end{lem}
\begin{proof}
In $F^n$ any element $a=(a_1,\ldots,a_n)$ with distinct $a_1, \ldots,a_n\in \mg{F}$ is invertible and has minimal polynomial
  $\prod_{i=1}^n(X-a_i)$ over $F$, which is separable of degree~$n$, whereby $L=F[a]$.
\end{proof}

Let $A$ be an $F$-algebra. We denote by $Z(A)$ the centre of $A$ and by $A^\op$ the opposite algebra of $A$.

\begin{lem}\label{lem:module-centre}
Let $K=Z(A)$ and assume that $K$ is an \'etale $F$-algebra.
Let $L$ be a commutative semisimple $F$-subalgebra of $A$ which is $F$-linearly disjoint from $K$.
Then $A$ is free as a left (resp.~right) $L$-module if and only if $A$ is free as a left (resp.~right) $KL$-module.
\end{lem}
\begin{proof}
We prove the statement for left modules, the proof for right modules is analogous.
Note that the commutative $L$-algebra $KL$ is isomorphic to $K\otimes_F L$, which is free as an $L$-module. 
Hence, if $A$ is free as a left $KL$-module, then it is free as a left $L$-module.

Suppose conversely that $A$ is free as a left $L$-module.
Then $K\otimes_F A$ is free as a left
$K\otimes_FL$-module.
We have $K\otimes_FK\simeq K^{[K:F]}$ as $K$-modules and thus obtain isomorphisms of left $K\otimes_FL$-modules
$$K\otimes_FA\simeq K\otimes_F K\otimes_K A \simeq K^{[K:F]}\otimes_KA\simeq A^{[K:F]}\,.$$
Identifying $K\otimes_F L$ with $KL$ we conclude that $A^{[K:F]}$ is 
free as a $KL$-module.
We will show that this is only possible if $A$ itself is free as a left $KL$-module.

Since $K$ is \'etale and $F$-linearly disjoint from $L$, it follows from \cite[Chap.~V, \S6, N$^\circ$~7]{Bou-A4-7} that $KL\simeq K_1\times \dots\times K_r$ for some  fields $K_1,\dots,K_r$. 
Consider a finitely generated module $M$ over $K_1\times \dots\times K_r$. Then $M$ is of the form $M_1\times \dots \times M_r$ where $M_i$ is a $K_i$-vector space for $i=1,\dots,r$. Furthermore $M$ is free if and only if the dimensions $\dim_{K_i}M_i$ for $i=1,\dots,r$ are all the same.
In particular, $M^n$ is free for an arbitrary positive integer $n$ if and only if $M$ is free.
\end{proof}

We call the $F$-algebra $A$ \emph{central simple} if $\dim_FA<\infty$, $Z(A)=F$ and $A$ is simple as a ring.
Let $A$ be a finite-dimensional simple $F$-algebra.
Then $K=Z(A)$ is a field and $A$ is a central simple $K$-algebra.
By Wedderburn's Theorem (cf.~\cite[Theorem~2.1.3]{GS06}) we have $\dim_KA=n^2$ for some positive integer $n$, which is called the \emph{degree of $A$} and denoted by $\deg A$.
Moreover, $A$ is Brauer equivalent to a central division $K$-algebra $D$, which is is unique up to $K$-isomorphism. The degree of $D$ is called the \emph{index of $A$} and denoted by $\ind A$.
If $\ind A=1$ then $A$ is $K$-isomorphic to $\matr{n}(K)$ for $n=\deg A$, and in this case we say that $A$ is \emph{split}.
We further set $\coind A=\frac{\deg A}{\ind A}$ and call this the \emph{coindex of $A$}.
Hence, for any finite-dimensional division $F$-algebra $D$ and any positive integer $n$ we have $\coind\matr{n}({D})=n$.

Let $A$ be an $F$-algebra.
For any $F$-subalgebra $B$ of $A$ we obtain an $F$-subalgebra
$$C_A(B)=\{x\in A\mid xb=bx\mbox{ for all }b\in B\},$$
called the \emph{centraliser of $B$ in $A$}.

An element $e\in A$ is called an \emph{idempotent} if $e^2=e$.
For any nonzero idempotent $e\in A$ the  ring $eAe$  with unity $e$ becomes an $F$-algebra by identifying $F$ with $eF$. Moreover, if $A$ is a central simple $F$-algebra, then the $F$-algebra $eAe$ is also central simple, and it is Brauer equivalent to $A$.

A crucial tool in the study of central simple algebras and their simple subalgebras is the Double Centraliser Theorem.
We refer to ~\cite[Sect.~12.7]{Pierce} for the statement.
The following is an extension of this statement for the case of commutative subalgebras.
\begin{prop}
  \label{prop:neat}
  Let $A$ be a finite-dimensional simple $F$-algebra.
  Assume that $Z(A)$ is separable over~$F$. Let $L$ be a commutative semisimple $F$-subalgebra of $A$ that is $F$-linearly disjoint from $Z(A)$.
  Then $C_A(L)$ is a
  semisimple $F$-algebra with centre ~$L$ and
   \[ [L:F]\cdot\dim_FC_A(L) \geq \dim_FA.\]
  Moreover, the following conditions are equivalent:
  \begin{enumerate}[$(a)$]
  \item
  $[L:F]\cdot\dim_F C_A(L)=\dim_F A$;
  \item
  all simple components of $C_A(L)$ have the same degree;
  \item
  $A$ is free as a left $L$-module;
 \item
  $A$ is free as a right $L$-module.
  \end{enumerate}
  They hold in particular whenever $L$ is a field or $[L:F]=\deg A$.
\end{prop}

\begin{proof}
Let $K=Z(A)$.
Note that $[KL:K]=[L:F]$, $C_A(KL)=C_A(L)$ and further that $\dim_FC_A(L)=[K:F]\cdot \dim_KC_A(KL)$ and $\dim_FA=[K:F]\cdot \dim_KA$.
In view of the statement and of Lemma~\ref{lem:module-centre}, we may therefore replace $K$ by $F$ and $LK$ by $L$. Hence we may assume in the sequel that $A$ is central simple as an $F$-algebra.

If now $L$ is a field, then $(a)$ holds by the Double
Centraliser Theorem and furthermore conditions $(b)$--$(d)$ are
trivially satisfied. This case will be used to show the statement in general.

More generally, let $e_1$, \ldots, $e_r$ be the primitive idempotents of $L$. For
  $i=1, \ldots, r$ we set $A_i=e_iAe_i$ and $L_i=e_iL$. Thus,
  identifying $F$ with $Fe_i\subseteq L_i$, each $L_i$ is a finite 
  field extension of $F$ contained in the central simple $F$-algebra
  $A_i$, and in the decomposition $A=\bigoplus_{i,j=1}^r e_iAe_j$ we
  have
  \[
  C_A(L)=C_{A_1}(L_1)\oplus\cdots\oplus C_{A_r}(L_r).
  \]
  Each $C_{A_i}(L_i)$ is a simple $F$-algebra with centre  $L_i$, whereby
  $C_A(L)$ is a semisimple algebra with centre ~$L$.
For $i=1,\dots,r$ we set $\ell_i=[L_i:F]$ and $d_i=\deg C_{A_i}(L_i)$ and obtain from the Double Centraliser Theorem that $$\dim_F A_i=  [L_i:F]\cdot \dim_F C_{A_i}(L_i)=\ell_i^2d_i^2,$$ whereby $\deg A_i=\ell_id_i$.
It follows that $\deg A=\sum_{i=1}^r\ell_id_i$ and
  \begin{equation*}
    \dim_F C_A(L) = \sum_{i=1}^r\, \dim_F C_{A_i}(L_i)= \sum_{i=1}^r \ell_id_i^2.
  \end{equation*}
As  $[L:F]=\sum_{i=1}^r \ell_i$ it 
  follows that
 \begin{eqnarray*}
  [L:F]\cdot\dim_FC_A(L) &= & \Bigl(\sum_{i=1}^r\ell_i\Bigr) \cdot
  \Bigl(\sum_{i=1}^r\ell_id_i^2\Bigr) \\
  & = &
  \Bigl(\sum_{i=1}^r\ell_id_i\Bigr)^2 + \sum_{i<j}
  \ell_i\ell_j(d_i-d_j)^2\\
 & =  &(\deg A)^2 + \sum_{i<j}
  \ell_i\ell_j(d_i-d_j)^2.
  \end{eqnarray*}
  This proves the inequality in the statement as well as the equivalence
  of~$(a)$ and $(b)$ because the last term on the right hand side vanishes if
  and only if $d_1=\cdots=d_r$. To prove the equivalence of $(b)$ with
  $(c)$, note that for $i=1, \ldots, r$ we have
  \[
  \dim_F e_iA  =\deg A_i \cdot\deg A =\ell_id_i\deg A, 
  \]
as one sees easily by reduction to the split case,  and thus
   $\dim_{L_i} e_iA =d_i\deg A $.
  Hence Condition~$(b)$ holds if and only if
  $\dim_{L_1} e_1A =\cdots=\dim_{L_r} e_rA$, which is 
 Condition~$(c)$.
The proof of the equivalence of $(b)$ with $(d)$ is completely analogous.

Finally, if 
$[L:F]=\deg A$, then
 $\sum_{i=1}^r\ell_i=[L:F]=\deg A=\sum_{i=1}^r\ell_id_i$, and thus $d_1=\dots=d_r=1$, which implies Condition $(b)$.
\end{proof}
Let $K$ be a quadratic \'etale $F$-subalgebra of $A$ and let $\gamma$ denote its nontrivial $F$-automorphism.
We denote
$$C'_A(K)=\{x\in A\mid xk=\gamma(k)x\mbox{ for all }k\in K\}\,.$$

If $A$ is a finite-dimensional semisimple $F$-algebra then for an element $a\in A$ we denote by $$\Prd_{A,a}(X)\in F[X]$$
its reduced characteristic polynomial (see \cite[\S9]{Reiner}).

\begin{prop}
  \label{lem:PC}
Assume that $A$ is a central simple $F$-algebra. Let $K$ be an $F$-subalgebra of $A$ isomorphic to $F\times F$.
Let $e_1$ and $e_2$ be the primitive idempotents of $K$ and $A_i=e_iAe_i$ for $i=1,2$.
Suppose that $\deg A_1 =\deg A_2$.
Let $a\in C'_A(K)$, $u=e_1ae_2$ and $v=e_2ae_1$.
Then $a=u+v$ and
  $$\Prd_{A,a}(X)=\Prd_{A_1,uv}(X^2)=\Prd_{A_2,vu}(X^2)\in F[X]\,.$$
\end{prop}

\begin{proof}
We have that $ae_1=e_2a$ and $ae_2=e_1a$, hence
$e_1a=e_1ae_2=u$ and $e_2a=e_2ae_1=v$, which yields that
$a=e_1a+e_2a=u+v$,
\[
uv=e_1ae_2ae_1=e_1a^2e_1\in A_1 \quad\text{and}\quad
vu=e_2ae_1ae_2=e_2a^2e_2\in A_2.
\]

To prove the equalities we may extend scalars to a splitting field of
  $A$. Thus we assume that $A$ is split and identify $A$
  with a matrix algebra in such a way that 
  $e_1=\bigl(
  \begin{smallmatrix}
    1&0\\0&0
  \end{smallmatrix}
  \bigr)$, $e_2=\bigl(
  \begin{smallmatrix}
    0&0\\0&1
  \end{smallmatrix}
  \bigr)$ and $a=\bigl(
  \begin{smallmatrix}
    0&u\\ v&0
  \end{smallmatrix}
  \bigr)$, where $u$ and $v$ are matrices. 
  We have to show that the characteristic polynomials
  $\Pc_a$, $\Pc_{uv}$, $\Pc_{vu}$ 
  are related by
  \begin{equation*}
    \label{eq:Pc}
    \Pc_a(X)=\Pc_{uv}(X^2)=\Pc_{vu}(X^2).
  \end{equation*}
  Since the coefficients of the characteristic polynomials are
  polynomial functions of the entries, it suffices to
  prove these equalities in the case where 
   $u$ and $v$ are generic matrices over
  $\mathbb{Z}$, for the general case 
   then follows by
  specialization. 
Since we have  $\Tr(a^{2k})=2\Tr\bigl((uv)^{k}\bigr)=2\Tr\bigl((vu)^{k}\bigr)$ and $\Tr(a^{2k+1})=0$  for any $k\in\nat$,
 we obtain the result by applying Newton's Identities relating the coefficients of the
  characteristic polynomial of a matrix to the traces of its powers.
\end{proof}

\begin{cor}
  \label{lem:PC2}
 Assume that $A$ is a central simple $F$-algebra. Let $K$ be an  \'etale quadratic $F$-subalgebra of $A$ such that $\dim_FC_A(K)=\frac12\dim_FA$.
  Then 
  $$\Prd_{A,a}(X)=\Prd_{C_{A}(K),a^2}(X^2)\in F[X]\mbox{ for any }a\in C'_A(K)\,.$$ 
\end{cor}

\begin{proof}
To prove the equality we may extend scalars. Hence we may assume that $K\simeq F\times F$.
Then the equation follows from
  Proposition~\ref{lem:PC}. 
\end{proof}

\section{Capacity} 
\label{S:cap}

In this section we recall some basic facts and objects associated with involutions on central simple algebras.
We recall the distinction of involutions into two kinds 
and into three different types. 
We further introduce some notation that will allow us to study involutions of different types and over fields of arbitrary characteristic in a unified way.
Our main reference for involutions is \cite{BOI}.
\medbreak

Let $A$ be an $F$-algebra.
By an \emph{$F$-involution on $A$} we mean an $F$-linear anti-automorphism $\sigma:A\to A$ such that $\sigma\circ\sigma=\id_A$.
Given an $F$-involution on $A$ we set
\begin{eqnarray*}
\Symm(\sigma) & = & \{x\in A\mid \sigma(x)=\phantom{-}x\}\,,\\
\Skew(\sigma) & = & \{x\in A\mid \sigma(x)=-x\}\,,\\
\Symd(\sigma) & = & \{x+\sigma(x)\mid x\in A\}\,.
\end{eqnarray*}

By an \emph{$F$-algebra with involution} we mean a pair $(A,\s)$ of a finite-dimensional $F$-algebra $A$ and an $F$-involution $\s$ on $A$ with 
$F= Z(A)\cap \Symm(\s)$ and such that
$A$ has no non-trivial two-sided ideal $I$ with $\s(I)=I$.

In the sequel, let $(A,\s)$ denote an $F$-algebra with involution.
Then either $Z(A)=F$ or
$Z(A)$ is a quadratic \'etale extension of $F$ with non-trivial automorphism $\s|_{Z(A)}$.
One says that $(A,\s)$, or the involution $\s$, is of the \emph{first kind} or of the \emph{second kind}, respectively, according to whether $[Z(A):F]$ equals $1$ or $2$.

As long as $Z(A)$ is a field it follows that $A$ is central simple as a $Z(A)$-algebra.
However, if $(A,\s)$ is of the second kind, we may also have that $Z(A)\simeq F\times F$: in this case $(A,\s)\simeq (A_0\times A_0^\op,\sw)$ for a central simple $F$-algebra $A_0$ and  where
$\sw$ is the so-called \emph{switch-involution} given by
$\sw(a_1,a_2^\op)=(a_2,a_1^\op)$ (see \cite[(2.14)]{BOI}). 

If $\s$ is an involution of the first kind, then we say that $\s$ is \emph{symplectic} if $\dim_F\Symd(\s)<\dim_F\Skew(\s)$ and $1\in\Symd(\s)$, otherwise we say that $\s$ is \emph{orthogonal}.
Considering the $F$-linear map
$x\mapsto x+\s(x)$ one sees that $$\dim_F\Skew(\s)+\dim_F\Symd(\s)=\dim_FA,$$ hence
$\dim_F\Symd(\s)<\dim_F\Skew(\s)$ if and only if $\dim_F\Symd(\s)<\frac12\dim_FA$.

If $\s$ is of the second kind then we also say that $\s$ is \emph{unitary}.
We say that $(A,\s)$ is \emph{unitary of inner type} when $Z(A)\simeq F\times F$.
(The term is motivated by a corresponding notion for algebraic groups.)

The property of the involution $\s$ to be \emph{orthogonal}, \emph{symplectic} or \emph{unitary} is called its \emph{type}.
Notions for properties of an involution (such as its kind and its type) shall also be employed for the algebra with involution as a pair.

Whenever $Z(A)$ is a field we denote by $\deg A$, $\ind A$, $\coind A$ the degree, index or coindex of $A$, respectively, as a central simple $Z(A)$-algebra.
In the case where $Z(A)\simeq F\times F$, we define the same terms with reference to (any of) the two simple components of $A$.
We say that the algebra with involution $(A,\s)$ is \emph{split} if $\ind A=1$.

We have
$\Symd(\s)\subseteq \Symm(\s)$, 
and this is an equality unless $\car F = 2$ and $(A,\s)$ is of the first kind. (See
\cite[(2.17)]{BOI} for $\car F =2$ and $\sigma$  unitary.)
To avoid case distinctions in our statements and arguments, we set
\[
\Syms(\sigma)=
\begin{cases}
  \Symm(\sigma)&\text{if $\sigma$ is orthogonal or unitary,}\\
  \Symd(\sigma)&\text{if $\sigma$ is symplectic.}
\end{cases}
\]
Note that $\Syms(\s)=\Symd(\s)$ except when $\car F =2$ and $\s$ is orthogonal.

Dealing with orthogonal involutions in characteristic two requires additional care, as one may see in the following statement.

\begin{prop}\label{prop:keepstype}
Let $e$ be a nonzero idempotent in $\Symm(\s)$ and $\s_e=\s|_{eAe}$.
Then $(eAe,\s_e)$ is an $F$-algebra with involution of the same kind as $(A,\s)$.
Moreover, $(eAe,\s_e)$ is of the same type as $(A,\s)$ except when $\car F =2$, $\s$ is orthogonal and $e\in\Symd(\s)$, in which case $\s_e$ is symplectic.
\end{prop}

\begin{proof}
Obviously the $F$-algebra $eAe$ is stable under the $F$-involution $\s_e$.
Let $K=Z(A)$.
We first show that $Z(eAe)=Ke$.
For this we may assume that $A$ is split and identify $A$ with $\End_K(V)$ for a finitely generated
free $K$-module $V$.
Then $eAe$ is identified with $\End_K(eV)$.
Moreover, the $K$-submodule $eV$ of $V$ is free; if $K$ is a field this is trivial, and otherwise we have $(A,\s)\simeq (A_0\times A_0^\op,\sw)$ for a central simple $F$-algebra $A_0$ and 
use that $e\in\Symm(\s)$ to obtain this conclusion.
Hence the centres of $\End_K(V)$ and $\End_K(eV)$ consist of the
scaling maps with scalars from $K$. 
This naturally identifies $Z(eAe)$ with $Ke$.

Hence $(eAe,\s_e)$ is an $F$-algebra with involution
of the same kind as $(A,\s)$.
To compare the types of the involutions, we only need to consider the case where $(A,\s)$ is of the first kind.

  If $a\in A$ is such that $a+\sigma(a)$ equals $1$ or $e$, then $eae+\sigma(eae)=e$. 
Hence, if $1\in \Symd(\s)$ or $e\in \Symd(\s)$, then $e\in\Symd(\sigma_e)$. It remains to consider the dimensions of $\Symd(\s)$ and $\Symd(\s_e)$.

  Let $f=1-e$, and let $n=\deg A$,
  $r=\deg eAe$, and $s=\deg fAf$, so that $n=r+s$. The decomposition
  $$A=eAe\oplus(eAf\oplus fAe)\oplus fAf$$ is stable under $\sigma$,
  hence
  \begin{equation*}
    \Symd(\sigma) = \bigl(eAe\cap\Symd(\sigma)\bigr) \oplus
    \bigl((eAf\oplus fAe)\cap\Symd(\sigma)\bigr) \oplus
    \bigl(fAf\cap\Symd(\sigma)\bigr).
  \end{equation*}
  If $a\in A$ satisfies $a+\sigma(a)\in eAe$, then
  $a+\sigma(a)=eae+\sigma(eae)$.
  This shows that 
  \[
  eAe\cap\Symd(\sigma) = \Symd(\sigma_e),
  \]
  and it follows that
  $\dim_F\bigl(eAe\cap\Symd(\sigma)\bigr)=\frac12r(r+\varepsilon_e)$
  with $\varepsilon_e=\pm1$. Likewise,
  $\dim_F\bigl(fAf\cap\Symd(\sigma)\bigr)=\frac12s(s+\varepsilon_f)$ with
  $\varepsilon_f=\pm1$. Now, if $a\in A$ is such that $a+\sigma(a)=ebf+fce$
  for some $b$, $c\in A$, then 
  \[
  ebf+fce=\sigma(ebf+fce)=e\sigma(c)f+f\sigma(b)e,
  \]
  hence $fce=f\sigma(b)e=\sigma(ebf)$ and $a+\sigma(a)=ebf+\sigma(ebf)$. Therefore
  \[
  (eAf\oplus fAe)\cap\Symd(\sigma) = \{x+\sigma(x)\mid x\in eAf\},
  \]
  and it follows that $\dim_F (eAf\oplus fAe)\cap\Symd(\sigma)=\dim_F
  eAf=rs$. Therefore the above decomposition of $\Symd(\s)$ 
   yields
\[
  \dim_F\Symd(\sigma)=\textstyle{\frac12}r(r+\varepsilon_e) +rs +
    \textstyle{\frac12}s(s+\varepsilon_f) = \textstyle{\frac12}(n^2+r\varepsilon_e+s\varepsilon_f). 
\]
As $\dim_F\Symd(\sigma)=\frac12n(n+\varepsilon)$ for $\varepsilon=\pm 1$, we conclude that $\varepsilon=\varepsilon_e=\varepsilon_f$.
Hence, $\dim_F\Symd(\s)<\frac12 \dim_FA$ if and only if $\dim_F\Symd(\s_e)<\frac12 \dim_FeAe$.
 \end{proof}

We give an example for the exceptional case in the statement of Proposition~\ref{prop:keepstype}.

\begin{ex}
  \label{ex:orthrestricttosymp}
  Write $t$ for the transpose involution on
  $\matr{4}(F)$.
 Consider the  matrices
  \[
  m=
  \begin{pmatrix}
   1&0&0&0\\
   0&1&0&0\\
   0&0&0&1\\
   0&0&1&0
  \end{pmatrix}\,\,\mbox{ and }\,\, e= \begin{pmatrix}
   0&0&0&0\\
   0&0&0&0\\
   0&0&1&0\\
   0&0&0&1
  \end{pmatrix}
  \]
  in $\matr{4}(F)$.
  The involution $\sigma=\Int(m)\circ t$ is orthogonal because $m$ is
  not alternating.
It is further easy to see that $e\in\Symd(\s)$. 
Hence, if $\car(F)=2$, 
we obtain that $\s$ restricts to a symplectic involution on $e\matr{4}(F)e$.
\end{ex}

We define
\[
\kap(A,\sigma)=\left\{\begin{array}{rl}
\deg A & \mbox{if $\sigma$ is orthogonal or unitary,}\\
\frac{1}{2}\deg A & \mbox{if $\sigma$ is symplectic,}
\end{array}\right.
\]
and we call this integer the \emph{capacity of $(A,\sigma)$}. 
This terminology is inspired by the theory of Jordan algebras: when the
characteristic is different from~$2$ and the algebra $A$ is split, then
$\Symm(\sigma)$ is a Jordan algebra of capacity equal to $\kap(A,\sigma)$; see~\cite[\S~I.5.1]{McC}.
Note that with this definition there exist $F$-algebras with
involution of any given type and any positive integer as capacity.

\begin{prop}\label{P:inv-sym-restrict-type}
Let $L$ be an $F$-subalgebra of $A$ contained in $\Symm(\s)$ and such that $L/F$ is a separable field extension. Set $C=C_A(L)$ and $\s_C=\s|_C$.
Then $(C,\s_C)$ is an $L$-algebra with involution of the same type as $(A,\s)$ and such that 
$$\kap(A,\s)=[L:F]\cdot\kap(C,\s_C)\,.$$
\end{prop}

\begin{proof}
With $\s(L)=L$ we also have that $\s(C)=C$.
Let $K=Z(A)$. Then $[LK:L]=[K:F]$ and $[LK:K]=[L:F]$.
The Double Centraliser Theorem yields that $C_A(C)=KL$
and $\deg A=[L:F]\cdot \deg C$.
As obviously $KL\subseteq Z(C)\subseteq C_A(C)$
we conclude that $Z(C)=KL$.
Hence, $(C,\s_C)$ is an $L$-algebra with involution,
and $\s_C$ is unitary if and only if $\s$ is unitary.
Using \cite[(4.12)]{BOI} in the cases where $\s$ is of the first kind, we conclude that 
the $L$-algebra with involution $(C,\s_C)$ has the same type as the $F$-algebra with involution $(A,\s)$.
Since $\deg A=[L:F]\cdot \deg C$, this implies the claimed equality for the capacity.
\end{proof}

We are going to show in Theorem~\ref{thm:capmaxdim} that the capacity of $(A,\s)$ is equal to the maximal degree $[L:F]$ where $L$ is an \'etale $F$-algebra contained in $\Syms(\s)$. 
To this end we first consider the case where $A$ is split and show that we even find then a split \'etale subalgebra  in $\Syms(\s)$ of degree equal to $\kap(A,\s)$.

\begin{prop}
  \label{lem:capmaxdimsplit}
Let $d=\kap(A,\s)$ and assume that $(A,\s)$ is split.
Then $\Syms(\sigma)$ contains an $F$-subalgebra $L$ of $A$ with $L\simeq F^d$.
\end{prop}

\begin{proof}
Assume first that $Z(A)$ is not a field. Then $(A,\sigma)$ can be identified with $(\matr{d}(F)\times \matr{d}(F)^\op,\sw)$. Letting $L_0\subseteq
  \matr{d}(F)$ be the algebra of diagonal matrices and $L=\{(x,x)\mid
  x\in L_0\}$, we obtain that $L\subseteq \Syms(\sigma)$ and $L\simeq F^d$. 
  
Assume now that $K=Z(A)$ is a field.
We identify $A$ with $\End_{K}V$ for some
  $K$-vector space $V$. Then $\sigma$ is the
  adjoint involution of some nondegenerate $F$-bilinear form $b:V\times V\to K$, which is symmetric and non-alternating if $\sigma$ is  orthogonal, which is alternating if
  $\sigma$ is symplectic, and which is hermitian with respect to the nontrivial $F$-automorphism of $K$ if $\sigma$ is unitary. If $\sigma$ is orthogonal or unitary (resp.\
  symplectic), we  have $\dim_KV=d$ (resp. $\dim_KV=2d$), and we obtain a decomposition of $V$ into a direct sum of $1$-dimensional (resp.\ $2$-dimensional) $K$-subspaces
  \[
  V=V_1\oplus\ldots\oplus V_d
  \]
that is an orthogonal decomposition for $b$.
Let $e_1,\dots,e_d$ denote the orthogonal projections corresponding to this decomposition.
Then $e_1,\dots,e_d$ are idempotents in $A=\End_KV$.
For $i=1,\dots,d$ and $x,y\in V$ we have
\[
  b\bigl(x,e_i(y)\bigr) = b\bigl(e_i(x),e_i(y)\bigr) = b(e_i(x),y).
  \]
Thus $e_1,\dots,e_d\in \Symm(\s)$, and we conclude by Proposition~\ref{prop:keepstype} that $e_1,\dots,e_d\in\Syms(\s)$.
Hence $L=Fe_1\oplus\dots\oplus Fe_d$ is an $F$-subalgebra of $A$ contained in $\Syms(\sigma)$ with $L\simeq F^d$.
\end{proof}

\section{Forms on the space of symmetrized elements} 

Certain statements on the existence for elements or subalgebras with special properties in an algebra with involution can be proven by reducing to the situation where the base field is algebraically closed.
This requires a geometric description of the property in question.
Here we are interested in elements and subalgebras contained in $\Syms(\s)$.
To obtain a geometric formulation we introduce a polynomial $\chi_a\in F[X]$ associated to an arbitrary element $a\in \Syms(\s)$, whose degree is equal to $\kap(A,\s)$ and which has $a$ as a root.
It is defined as either the
reduced characteristic polynomial $\Prd_{A,a}$ or the Pfaffian
characteristic polynomial $\Prp_{\sigma,a}$ (see
\cite[(2.10)]{BOI}):
\begin{equation*}
  \chi_{a}=
  \begin{cases}
    \Prd_{A,a}&\text{if $\sigma$ is orthogonal or unitary,}\\
    \Prp_{\sigma,a}&\text{if $\sigma$ is symplectic.}
  \end{cases}
\end{equation*}
For the unitary case, note that, even though the coefficients of the
reduced characteristic polynomial of any $a\in A$ lie in $Z(A)$, when
$\sigma(a)=a$ the coefficients of $\Prd_{A,a}$ lie in $F$ (see
\cite[(2.16)]{BOI}). (When $(A,\sigma)=(A_0\times A_0^\op,\sw)$, then
$a=(a_0,a_0^\op)$ for some $a_0\in A_0$, and $\chi_a=\Prd_{A_0,a_0}$.)
Thus we have $\chi_a\in F[X]$ in all cases. Note that 
$\chi_a$ is a multiple of the minimal polynomial of $a$ over $F$ and
that the two polynomials have the same irreducible factors. Therefore,
if $\chi_a$ is separable then $\chi_a$ is the minimal polynomial of
$a$ over $F$.

\begin{thm}
  \label{thm:capmaxdim}
Any \'etale $F$-subalgebra of $A$ contained in $\Symm(\sigma)$ is contained in $\Syms(\sigma)$.
Furthermore  \[
  \kap(A,\sigma) = \max\{ [L:F] \mid \text{$L$ \'etale $F$-algebra with
    $L\subseteq \Symm(\sigma)$}\}.
  \]
\end{thm}

\begin{proof}
  Let $L\subseteq \Symm(\sigma)$ be an \'etale $F$-algebra. To show that $L\subseteq \Syms(\s)$ and 
  $[L:F]\leq\kap(A,\sigma)$, we may extend scalars and assume that $F$ is
  algebraically closed. Then $L$ and $Z(A)$ are split. Let $r=[L:F]$ and let $e_1, \ldots, e_r\in L$ be the
  primitive idempotents in $L$. Then $e_ie_j=\delta_{ij}$ for $i,j\in\{1,\dots r\}$ and $\sum_{i=1}^r
  e_i=1$. It follows that
  \[
  \deg A=\sum_{i=1}^r\deg(e_iAe_i).
  \]
If $\s$ is orthogonal or unitary, then we have $\Syms(\s)=\Symm(\s)$ and further $\kap(A,\s)=\deg(A)\geq r$, because $\deg(e_iAe_i)\geq1$ for $i=1,\dots,r$.
Assume now that $\s$ is symplectic.
  Then Proposition~\ref{prop:keepstype} shows for $i=1,\dots,r$ that $\s$ restricts on $e_iAe_i$ to a symplectic  involution, whereby $\deg(e_iAe_i)\geq2$ and $e_i\in \Symd(\s)$. We conclude that $L\subseteq \Symd(\s)$ and $\kap(A,\sigma)=\frac12\deg A=\sum_{i=1}^r\frac12\deg e_iAe_i\geq r$.
  This shows that $L\subseteq\Syms(\s)$ and $[L:F]= r\leq \kap(A,\s)$ in any case.
   
Back in the situation where $F$ is an arbitrary field,
it remains to show that $\Syms(\sigma)$ contains
  an \'etale $F$-algebra $L$ with $[L:F]=\kap(A,\sigma)$. This follows
  from Proposition~\ref{lem:capmaxdimsplit} if $(A,\s)$ is split.
In particular, we may assume that $F$ is infinite.
Let $\overline F$ denote
  an algebraic closure of $F$. Then $\Syms(\sigma)$ is
  Zariski-dense in $\overline{\Syms(\sigma)}=\Syms(\sigma)\otimes_F\overline
  F$. 
An element $a\in \overline{\Syms(\sigma)}$ is separable over $\overline{F}$ if and only if the discriminant of $\chi_a$ is nonzero.
Since this is a polynomial condition,
the elements of $\overline{\Syms(\s)}$ which are separable over $\overline F$ form an open subset of $\overline{\Syms(\sigma)}$, and by Proposition~\ref{lem:capmaxdimsplit} and Lemma~\ref{L:split-etale-infinite} this subset is not empty.
Since $\Syms(\s)$ is dense in $\overline{\Syms(\s)}$, we conclude that there exists an element $a\in \Syms(\s)$ which is separable over $\overline{F}$, and thus separable over $F$. Hence $\chi_a$ is equal to the minimal polynomial of $a$ over $F$.
We conclude that  $F[a]$ is an \'etale $F$-algebra and $[F[a]:F]=\deg(\chi_a)=\kap(A,\sigma)$.
\end{proof}

In the context of the last result we observe that $\Syms(\sigma)$ may contain (non-\'etale) commutative
  $F$-algebras $L$ with $[L:F]>\kap(A,\sigma)$. 
\begin{ex}
   Let $L_0$ be the $F$-subalgebra of $\matr{4}(F)$ consisting of the matrices
  \[ 
  \begin{pmatrix}
  a & 0 & b & c \\
  0 & a & d & e \\
  0 & 0 & a & 0 \\
  0 & 0 & 0 & a \\
  \end{pmatrix}
  \]
with $a,b,c,d,e\in F$.
 Then $L=\{(x,x)\mid x\in L_0\}$ is a $5$-dimensional commutative $F$-subalgebra of  
  $A=\matr{4}(F)\times \matr{4}(F)^\op$.
For the involution  $\s=\sw$ on $A$ we have that  $\kap(A,\s)=4$ and $L\subseteq\Symm(\s)=\Syms(\s)$. 
\end{ex}

\begin{prop}
Let $\Psi:(A,\s)\to (B,\tau)$ be a homomorphism of $F$-algebras with involution with $\kap(A,\s)=\kap(B,\tau)$.
Then $\chi_{\Psi(a)}=\chi_a$ holds for every $a\in\Syms(\s)$.
Furthermore, if $(A,\s)$ and $(B,\tau)$ are of the same type, then $\Psi$ is an isomorphism.
\end{prop}
\begin{proof}
The characteristic polynomial is invariant under algebra isomorphisms and under scalar extension.
Hence, the first part of the statement is obtained by extending scalars to an algebraically closure,  where it is easy to verify.
Finally, if $(A,\s)$ and $(B,\tau)$ are of the same type then $\dim_FA=\dim_FB$ and since $\Psi$ is injective, it follows that it is an isomorphism of algebras with involution. 
\end{proof}

We give some examples of split algebras with involution and embeddings between them.
For a matrix $\alpha$ with coefficients in a ring we denote by $\alpha^t$ the transpose matrix of $\alpha$.
If $m$ is a positive integer and $\alpha$ and $\beta$ are two $m\times m$ matrices over a ring, then we denote by $\alpha\times \beta$ the $2m\times 2m$ matrix $\left(\begin{smallmatrix}\alpha & 0 \\ 0 & \beta\end{smallmatrix}\right)$.

\begin{prop}\label{P:split-case-embeddings}
Let $m$ be a positive integer.
Let
$$ s:\matr{2m}({F})\to  \matr{2m}({F}), \left(\begin{smallmatrix}\alpha & \beta \\ \gamma & \delta\end{smallmatrix}\right)\longmapsto \left(\begin{smallmatrix}\hphantom{-}\delta^t & -\beta^t \\ -\gamma^t & \hphantom{-}\alpha^t \end{smallmatrix}\right)\,.$$
We have the following:
\begin{enumerate}[$(a)$]
\item $(\matr{2m}(F),s)$ is an $F$-algebra with symplectic involution of capacity~$m$.
\item $\Phi:(\matr{m}(F),t)\to  (\matr{2m}(F), s), \alpha\mapsto \alpha\times{\alpha}$ is a homomorphism of $F$-algebras with involution.
\item $\Psi:(\matr{m}(F)\times\matr{m}(F)^\op,\sw)\to (\matr{2m}(F),s), (\alpha,\beta^\op)\mapsto \alpha\times \beta^t$ is a homomorphism of $F$-algebras with involution. 
\end{enumerate}
\end{prop}
\begin{proof}
This is obvious.
\end{proof}

We consider the case $m=2$.

\begin{ex}\label{E:cap2-split}
Let $\mathbb{I}_2=\left(\begin{smallmatrix} 1&0\\ 0&1 \end{smallmatrix}\right) \in\matr{2}({F})\mbox{ and }
J=\left(\begin{smallmatrix} 0 & \mathbb{I}_2 \\ -\mathbb{I}_2 & 0\end{smallmatrix}\right) \in\matr{4}({F})\,.$
The involution $s:\matr{4}({F})\to  \matr{4}({F})$ of Proposition~\ref{P:split-case-embeddings} is given by $\Int(J)\circ t$, where $t$ is the transposition involution.
The $F$-space $\Symd(s)$ consists of the matrices
$$\left(\begin{matrix}
a & b & 0 & e\\
      c & d & -e& 0\\
      0 & f & a & c\\
      -f& 0 & b & d
\end{matrix}\right)$$
with $a,b,c,d,e,f\in F$.
For later use we note that the determinant of such a matrix is equal to $(ad-bc+ef)^2$.
\end{ex}

Let $d=\kap(A,\s)$. For $a\in \Syms(\s)$ we write 
\[
\chi_a=X^d-c_1(a)X^{d-1}+c_2(a)X^{d-2}-\cdots+(-1)^dc_d(a)
\]
and observe that this defines  a form $c_i\colon \Syms(\sigma)\to F$ of degree~$i$ for $i=1,\dots,d$.

\medskip
We recall some quadratic form terminology from \cite[(7.17)]{EKM}.
Let $q:V\to F$ be a quadratic form over $F$, defined on a finite-dimensional $F$-vector space $V$.
We denote by $b_q$ the \emph{polar form of~$q$} given by $$V\times V\to F, (x,y)\mapsto q(x+y)-q(x)-q(y)\,.$$
We further set 
\begin{eqnarray*}\rad(b_q) & = & \{x\in V\mid b_q(x,y)=0\mbox{ for all } y\in V\}\\
\rad(q) & = & \{x\in \rad(b_q)\mid q(x)=0\}\,
\end{eqnarray*}
and observe that these are $F$-subspaces of $V$ with $\rad(q)\subseteq \rad(b_q)$. 
Moreover,
if $\car F\neq 2$ then $q(x)=\frac{1}{2}b_q(x,x)$ for all $x\in V$ and thus $\rad(q)=\rad(b_q)$.
We call the quadratic form $q$ \emph{regular} if $\rad(q)=\{0\}$ and \emph{nondegenerate} 
 if $q$ is regular and $\dim_F\rad(b_q)\leqslant 1$.

\begin{prop}\label{P:cap2-c2-regquad}
Assume that $\kap(A,\s)=2$ and set $V=\Syms(\s)$.
Then
$$\dim_F V=\left\{\begin{array}{ll} 3 & \mbox{ if $\s$ is orthogonal},\\
4 & \mbox{ if $\s$ is unitary},\\
6 & \mbox{ if $\s$ is symplectic}
\end{array}\right.$$
and $c_2|_V:V\to F$ is a nondegenerate quadratic form over $F$, also given by the rule $x\mapsto x\ovl{x}$ where $\ovl{x}=c_1(x)-x$ for $x\in V$.
\end{prop}
\begin{proof}
By the definitions of the capacity and of $V$ the value of $\dim_FV$ follows from \cite[(2.6)]{BOI}.
For $x\in V$ we have
$c_2(x)=c_2(x)-\chi_x(x)=-x^2+c_1(x)x=x\ovl{x}$.

To show that the quadratic form $c_2$ is nondegenerate we may extend scalars and assume that $F$ is algebraically closed. 
Note that $(A,\s)$ is isomorphic to any $F$-algebra with involution of same type and of capacity~$2$.
It thus suffices to prove that $c_2$ is nondegenerate for a convenient choice of $(A,\s)$. 

Consider the $F$-linear map $$\Gamma:F^6\to \matr{4}({F})\,,(a,b,c,d,e,f)\longmapsto\left(\begin{matrix}
a & b & 0 & e\\
      c & d & -e& 0\\
      0 & f & a & c\\
      -f& 0 & b & d
\end{matrix}\right)\,.$$
The characteristic polynomial of $\Gamma(a,b,c,d,e,f)$ is
$(T^2-(a+d)T+(ad-bc+ef))^2$, hence its Pfaffian polynomial is $T^2-(a+d)T+(ad-bc+ef)$.
In particular, $c_2(\Gamma(a,b,c,d,e,f))=ad-bc+ef$.

Suppose that $\s$ is symplectic.
Then $(A,\s)$ is identified with $(\matr{4}(F),s)$, whereby $V=\Gamma(F^6)$.
Hence the  form $c_2$ is given by the polynomial $X_1X_4-X_2X_3+X_5X_6$, thus it is hyperbolic and in particular nondegenerate.

Suppose that $\s$ is unitary.
Then $(A,\s)$ is identified with the image of $\Psi$, thus $V=\{\Gamma(a,b,c,d,0,0)\mid a,b,c,d\in F\}$.
Hence, $c_2$ is given by the polynomial $X_1X_4-X_2X_3$, thus it is hyperbolic and in particular nondegenerate.

Suppose that $\s$ is orthogonal. 
Then $(A,\s)$ is identified with the image of $\Phi$. Thus we have $V=\{\Gamma(a,b,b,d,0,0)\mid a,b,d\in F\}$. 
Hence, $c_2$ is given by the polynomial $X_1X_4-X_2^2$, thus it is nondegenerate.
\end{proof}

\section{Neat subalgebras} 
\label{S:neat}

Let $(A,\sigma)$ be an $F$-algebra with involution. In this section we study \'etale subalgebras of $A$ that are contained in $\Symm(\s)$.
An $F$-subalgebra
$L$ of $A$ is called \emph{neat in $(A,\s)$} or a \emph{neat
  subalgebra of $(A,\sigma)$} if $L$ is \'etale,
$L\subseteq \Symm(\sigma)$, $A$~is free as a left $L$-module and 
for each nonzero idempotent $e$ of $L$, the 
$F$-algebra with involution $(eAe,\sigma\rvert_{eAe})$ has the same type as $(A,\s)$.

\begin{ex}
Any separable field extension of $F$ contained in $\Symm(\s)$ is neat in $(A,\s)$; this follows by Proposition~\ref{P:inv-sym-restrict-type}, using further Proposition~\ref{P:inner-unitary-neat} below in the case where $(A,\s)$ is unitary of inner type.
\end{ex}

By Proposition~\ref{prop:keepstype}, in the definition of neatness the only case where the condition on the idempotents does not follow from the other conditions is when $\car F=2$ and $\s$ is orthogonal.

\begin{prop}\label{P:idempotsymd}
Assume that $\car F=2$ and that $\s$ is orthogonal.
Let $L$ be an \'etale $F$-subalgebra of $A$ with $L\subseteq \Symm(\s)$.
Then the following are equivalent:
\begin{enumerate}[$(i)$]
\item $\Symd(\s)$ contains no nonzero idempotent of $L$.
\item $\Symd(\s)$ contains no primitive idempotent of $L$.
\item For every nonzero idempotent $e$ of $L$, the involution $\s|_{eAe}$ is orthogonal.
\end{enumerate}
In particular, $L$ is neat in $(A,\s)$ if and only if  $A$ is free as a left $L$-module and any of the Conditions $(i)$--$(iii)$ is satisfied.
\end{prop}
\begin{proof}If a nonzero idempotent $e$ of $L$ is contained in $\Symd(\s)$, then there exists a primitive idempotent $e'$ of $L$ with $e'e=e'$, and writing $e=a+\s(a)$ with $a\in A$ we obtain that $e'=e'ae'+\s(e'ae')\in\Symd(\s)$.
Hence $(i)$ and $(ii)$ are equivalent.

For any nonzero idempotent $e\in L$, 
the involution $\s|_{eAe}$ on $eAe$ is symplectic if and only if $e\in\Symd(\s)$.
This shows that $(i)$ and $(iii)$ are equivalent.
\end{proof}

Neat subalgebras of algebras with unitary involution of inner type are described in the next proposition. 
  
\begin{prop}\label{P:inner-unitary-neat}
Assume that $(A,\s)=(A_0\times A_0^\op,\sw)$ for a central simple $F$-algebra $A_0$.
Any \'etale $F$-subalgebra $L_0$ of $A_0$ gives rise to an \'etale $F$-subalgebra $L=\{(x,x^\op)\mid x\in L_0\}$ of $A$ contained in $\Symm(\s)$ and isomorphic to $L_0$.
Conversely, any \'etale $F$-subalgebra $L$ of $A$ contained in $\Symm(\s)$ is obtained in this way.
Furthermore, $L$ is neat in $(A,\s)$ if and only if $A_0$ is free as an $L_0$-left module.
\end{prop}

\begin{proof}
This is obvious.
\end{proof}

The previous proposition allows us to reformulate the condition of neatness in different ways in all cases complementary to the case treated in Proposition~\ref{P:idempotsymd}.

\begin{prop}\label{P:neat-equiv}
Assume that $\car F\neq 2$ or that $\s$ is symplectic or unitary.
Let $L$ be an \'etale $F$-subalgebra of $A$ with
$L\subseteq \Symm(\sigma)$.
The following conditions are equivalent:
\begin{enumerate}[$(a)$]
\item $[L:F]\cdot\dim_F C_A(L)=\dim_F A$;
\item all simple components of $C_A(L)$ have the same degree;
\item $L$ is neat in $(A,\s)$.
\end{enumerate}
\end{prop}
\begin{proof}
In view of the hypothesis and Proposition~\ref{prop:keepstype}, $L$ is neat in $(A,\s)$ as soon as $A$ is free as a left $L$-module.
If $A$ is simple, then the equivalences follow directly from Proposition~\ref{prop:neat}.
If $A$ is not simple, then $(A,\s)\simeq (A_0\times A_0^\op,\sw)$ for a central simple $F$-algebra $A_0$ and the equivalences follow by Proposition~\ref{P:inner-unitary-neat} along with Proposition~\ref{prop:neat}.
\end{proof}

The following proposition shows that the notion of neat
  subalgebra is preserved under scalar extension.

\begin{prop}\label{prop:neat-ext}
Let $L$ be a commutative $F$-subalgebra of $(A,\s)$ and let $F'/F$ be a field extension.
Then $L\otimes_FF'$ is neat in $(A_{F'},\s_{F'})$ if and only if $L$ is neat in $(A,\s)$.
\end{prop}
\begin{proof}
Let $L'=L\otimes_FF'$, $\s'=\s_{F'}$ and $A'=A_{F'}$.
Clearly the $F'$-algebra $L'$ is \'etale resp.~contained in
$\Symm(\s')$ if and only if the $F$-algebra $L$ is
\'etale~resp.~contained in $\Symm(\s)$.
Note that $C_{A'}(L')=C_A(L)\otimes_FF'$.
Hence, if $\car F\neq 2$ or if $\s$ is symplectic or unitary, then the statement follows immediately by Proposition~\ref{P:neat-equiv}.

We may therefore assume that $\car F=2$ and that $\s$ is orthogonal. 
It follows by Proposition~\ref{prop:neat} that $A'$ is free as a left $L'$-module if and only if $A$ is free as a left
  $L$-module. Hence, it suffices to check the condition on the
  idempotents.
As $\Symd(\s)\subseteq \Symd(\s')$ and $L\subseteq L'$, if $\Symd(\s')$ does not contain any nonzero idempotents of $L'$, then $\Symd(\s)$ does not contain any nonzero idempotents of $L$.
Hence, if $L'$ is neat in $(A',\s')$, then $L$ is neat in $(A,\s)$.

To show the converse implication, we will first reduce the problem to the case where $F'/F$ is a Galois extension.
Observe that if $F'/F$ is either a purely transcendental extension or a purely inseparable algebraic extension,  then under scalar extension from $F$ to $F'$ every
  separable field extension of $F$ remains a field, whereby
  $L\otimes_FF'$ does not acquire new idempotents.
This observation allows us to reduce to the case where $F'/F$ is a separable algebraic extension.
Assuming now that $L'$ is not neat in $(A',\s')$, we may replace $F'$ by its Galois closure over $F$; by the implication that is already shown, the fact that $L'$ is not neat in $(A',\s')$ will be conserved.

With the assumption that $F'/F$ is a Galois extension, the Galois group acts naturally on $A'$ and on  $L'$ by fixing $A$
and $L$, respectively.
In view of Proposition~\ref{P:idempotsymd} we may choose a primitive idempotent $e'$ of $L'$ with $e'\in\Symd(\s')$.
 Let $e_1=e'$, $e_2$, \ldots, $e_r$ be the different  primitive idempotents obtained from $e'$ via the Galois action.
Let $e=e_1+\cdots+e_r$. Note that $e$ is fixed under the Galois action.
Since $L$ is the fixed field of the Galois action on $L'$, we conclude that $e$ is a nonzero idempotent of $L$.
As   $e'\in\Symd(\sigma')$ we have that $e_1,\dots,e_r\in\Symd(\sigma')$ and thus $e\in\Symd(\sigma')\cap A=\Symd(\sigma)$. This shows that $L$ contains a
  nonzero idempotent in $\Symd(\sigma)$. Hence $L$ is not neat in $(A,\s)$.
\end{proof}

We next show that any \'etale subalgebra contained in $\Symd(\s)$ and of maximal degree under this condition is neat.

\begin{prop}
  \label{prop:maxneat}
Let
  $L$ be an \'etale $F$-subalgebra of $A$ with $L\subseteq
  \Symm(\sigma)$ and such that $[L:F]=\kap(A,\sigma)$. Then $L$ is neat in
  $(A,\sigma)$. 
\end{prop}

\begin{proof}
By Proposition~\ref{prop:neat-ext}, in order to show that $L$ is neat we may extend scalars, hence we may assume
    that $L$ is split.
In the case where $(A,\s)$ is
    unitary of inner type, the statement    
     readily follows from
    Proposition~\ref{P:inner-unitary-neat} and 
    Proposition~\ref{prop:neat}. We may thus
    assume that $A$ is simple.

  Let $r=[L:F]$ and let $e_1$, \ldots, $e_r$ be the primitive idempotents of $L$.
Then $L=Fe_1\oplus\cdots\oplus Fe_r$. 
For $i=1$, \dots, $r$,
    set $A_i=e_iAe_i$, $d_i=\deg A_i$ and 
  $\sigma_i=\sigma\rvert_{A_i}$, whereby $(A_i,\sigma_i)$ is an
  $F$-algebra with involution with $Z(A_i)=e_iZ(A)e_i$ and
  $e_i\in \Syms(\sigma_i)$, according to
  Theorem~\ref{thm:capmaxdim}. 
By the hypothesis, we have that
  $$r= \kap(A,\sigma)=\frac1d\sum_{i=1}^rd_i$$
where $d=2$ if $\s$ is symplectic and $d=1$ otherwise.
  In the case where $\s$ is symplectic, we obtain for $i=1$, \dots,
  $r$ that $\s_i$ is symplectic and $e_i\in \Symd(\s_i)$, so that
  $\s_i$ restricts to a symplectic involution on $A_i$, whereby $d_i$ is even. 
  In any case we conclude from the above equality that
  $d_1=\dots=d_r=d$. Hence $A$ is free as a left $L$-module.

If $\car F\neq 2$ or if $\s$ is symplectic or unitary, then it follows by Proposition~\ref{P:neat-equiv} that $L$ is neat in $(A,\s)$.
Suppose that $\car F=2$ and that $\s$ is orthogonal.
Then $d_1=\dots=d_r=d=1$. It follows for $i=1,\dots,r$ that $\s_i$ is orthogonal and therefore $e_i\notin\Symd(\s_i)$.
Hence $e_1,\dots,e_r\notin\Symd(\s)$ and
we conclude by Proposition~\ref{P:idempotsymd} that $L$ is neat in $(A,\s)$.
\end{proof}

If $\car F=2$ and $\s$ is orthogonal then there may exist \'etale
$F$-subalgebras $L$ of $A$ that are maximal in $\Syms(\s)$ and with
$[L:F]<\kap(A,\s)$, as the following example illustrates. 
In any other case one can actually show that \'etale $F$-subalgebras
of $A$ contained in $\Syms(\s)$ and maximal for these properties are
of degree equal to $\kap(A,\s)$. 

\begin{ex}\label{ex:notneat}
Let $\car F=2$.
We enhance Example~\ref{ex:orthrestricttosymp}, where in the $F$-algebra $A=\matr{4}(F)$ we considered two matrices $m$ and $e$ and the orthogonal involution $\sigma=\Int(m)\circ t$ whose restriction to $eAe$ is symplectic.
Set 
 \[ e_1=  \begin{pmatrix}
     1&0&0&0\\
     0&0&0&0\\
     0&0&0&0\\
     0&0&0&0
   \end{pmatrix},\qquad
  e_2= \begin{pmatrix}
     0&0&0&0\\
     0&1&0&0\\
     0&0&0&0\\
     0&0&0&0
   \end{pmatrix}\,\,\in A.
  \]
 The $F$-subalgebra $L=Fe_1\oplus F e_2\oplus Fe$ of $A$ is split \'etale and maximal with respect to inclusion among
  the \'etale subalgebras of $\Symm(\s)=\Syms(\sigma)$, and yet we have $[L:F]=3<4=\kap(A,\sigma)$. The $F$-subalgebra $L$ is not neat in
  $(A,\sigma)$ because the restriction of $\sigma$ to $eAe$ is
  symplectic, but also because the conditions of
  Proposition~\ref{prop:neat} do not hold, since $[L:F]$ does
  not divide $[\matr{4}(F):F]$. 
Moreover, the $F$-algebra $L'=Fe_1\oplus F(e_2+e)$ is split \'etale and contained in $\Syms(\sigma)$, but even though
  the restrictions of $\sigma$ to $e_1Ae_1$ and $(e_2+e)A(e_2+e)$
  are orthogonal, $L'$ is not neat in $(A,\sigma)$ because
  the simple components of its centraliser do not have the same
  dimension. 
\end{ex}

Turning back to the situation where $(A,\s)$ is an arbitrary $F$-algebra with involution, our next goal is to characterize neat subalgebras as
  subalgebras of symmetric \'etale algebras of
  dimension~$\kap(A,\s)$.

\begin{lem}\label{lem:subneat}
Let $L$ be a neat $F$-subalgebra of $(A,\s)$ and let $K$ be an $F$-subalgebra of $L$.
If $L$ is free as a $K$-module, then $K$ is neat in $(A,\s)$.
\end{lem}
\begin{proof}
If $L$ is free as a $K$-module, then using that $A$ is free as a left
$L$-module we obtain that $A$ is free as a left $K$-module and
conclude that $K$ is neat in $(A,\s)$ since all the idempotents in $K$ are in $L$.
\end{proof}

\begin{thm}\label{thm:neat-embed-maxetsym}
Let $K$ be a commutative $F$-subalgebra of $(A,\s)$. 
Then $K$ is neat in $(A,\s)$ if and only if $K\subseteq
L\subseteq \Symm(\s)$ for some \'etale $F$-subalgebra $L$ of $A$ with
$[L:F]=\kap(A,\s)$ and such that $L$ is free as a $K$-module. 
Moreover, if $K$ and $(A,\s)$ are split, then one can choose $L$ to be split.
\end{thm}

\begin{proof}
Any \'etale $F$-subalgebra $L$ of $A$ with $L\subseteq \Symm(\s)$ and
$[L:F]=\kap(A,\s)$  is neat in $(A,\s)$,  by
Proposition~\ref{prop:maxneat}, and if $K$ is contained in such an
$F$-algebra $L$ which further is free as a $K$-module, then it follows by
Lemma~\ref{lem:subneat} that $K$ is neat in $(A,\s)$. 

Assume now that $K$ is neat in $(A,\s)$.
Let $e_1$, \ldots, $e_r$ be the primitive idempotents of $K$.
For $i=1$, \dots, $r$, set $K_i=e_i K$, $A_i=e_iAe_i$ and
$\sigma_i=\sigma\rvert_{A_i}$. 
Since $K$ is neat in $(A,\s)$, we have that $(A_i,\sigma_i)$ is an
$F$-algebra with involution of the same type as $(A,\s)$ and with
$Z(A_i)=e_iZ(A)e_i$. Moreover, since $A$ is free as a left
  $K$-module, all simple components $C_{A_i}(K_i)$ of $C_A(K)$
  have the same degree, by Proposition~\ref{P:neat-equiv} if $\car F\neq 2$ or $\s$ is symplectic or unitary and otherwise by Proposition~\ref{prop:neat}. 
  For all $i=1,\dots,r$ the $K_i$-algebras with involution $(C_{A_i}(K_i),\sigma\rvert_{C_{A_i}(K_i)})$  have the same capacity, which we denote by~$c$. Fix $i\in\{1,\dots,r\}$. By Theorem~\ref{thm:capmaxdim}
  there exists an \'etale $K_i$-subalgebra $L_i$ of $C_{A_i}(K_i)$
  such that $L_i\subseteq\Symm(\s\rvert_{C_{A_i}(K_i)})$ and $[L_i:K_i]=c$; moreover,
if $K$ and $(A,\s)$ are
  split, then $C_{A_i}(K_i)$ is split and $K_i\simeq F$, and we can choose $L_i$ to be a split \'etale $F$-algebra, by
  Proposition~\ref{lem:capmaxdimsplit}. 
By Proposition~\ref{P:inv-sym-restrict-type} we have $\kap(A_i,\sigma_i)=c\cdot [K_i:F]=[L_i:F]$.
Having this for $i=1$, \dots,  $r$, we obtain that $L=L_1\oplus\cdots\oplus L_r$ is an \'etale $F$-subalgebra of $(A,\s)$ contained in $\Symm(\s)$ and such that $[L:F]=\sum_{i=1}^r[L_i:F]=\sum_{i=1}^r\kap(A_i,\s_i)=\kap(A,\s)$. As a $K$-module $L$ is free because the dimensions $[L_i:K_i]$ are the same for all~$i$.
\end{proof}

\begin{cor}\label{C:neat-degrees}
Let $K$ be a neat $F$-subalgebra of $(A,\s)$.
Then $[K:F]$ divides $\kap(A,\s)$.
\end{cor}
\begin{proof}
This is obvious from Theorem~\ref{thm:neat-embed-maxetsym}.
\end{proof}

The following proposition shows how to construct split
  \'etale subalgebras in an $F$-algebra with involution represented as
  the endomorphism algebra of a hermitian or skew-hermitian space.
  We refer to \cite[\S4]{BOI} for the terminology and basic facts on hermitian forms.

\begin{prop}\label{prop:neat-hermitian}
Let $D$ be a finite-dimensional division $F$-algebra, $V$ a finite-dimensional right $D$-vector space and $A=\End_DV$. 
Let $\tau$ be an $F$-involution on $D$ for which $(D,\tau)$ is an $F$-algebra with involution. 
Let  $h:V\times V\to D$ be a hermitian or skew-hermitian form with respect to $\tau$ and let $\s$ be the $F$-involution on $A=\End_DV$ adjoint to $h$.
Let $V_1,\dots,V_r$ be $D$-subspaces of $V$
such that $$V=V_1\oplus \dots \oplus V_r\,.$$
Let $e_1,\dots,e_r \in A$ denote the projections corresponding to this decomposition and $L=Fe_1\oplus\dots\oplus Fe_r$.
Then the following hold:
\begin{enumerate}[$(i)$]
\item $L$ is a split $F$-\'etale subalgebra of $A$ with $[L:F]=r$.
\item $(A,\s)$ is an $F$-algebra with involution of the same kind as $(D,\tau)$.
\item  $L\subseteq \Symm(\s)$ if and only if
the decomposition $V=V_1\oplus \dots \oplus V_r$ is orthogonal with respect to $h$.
\item If $L\subseteq \Symm(\s)$, then $L$ is neat in $(A,\s)$ if and only if for $i=1,\dots,r$ we have $\dim_DV_i=\frac1r\dim_DV$ and $h$ restricts to a non-alternating form on $V_i$ in the case where $h$ is non-alternating.
\end{enumerate}
 \end{prop}
\begin{proof}
Part $(i)$ is clear and Part $(ii)$ is \cite[(4.2)]{BOI}.
By the definition of $L$ we have $L\subseteq\Symm(\sigma)$ if and only if 
 \[
  h\bigl(x,e_i(y)\bigr)=h\bigl(e_i(x),e_i(y)\bigr) = h(e_i(x),y)
  \]
holds for all $i=1, \ldots,r$ and $x,y\in V$, which is if and only if the decomposition of $V$ is orthogonal with respect to $h$. This shows $(iii)$.

We have 
  \[
  C_A(L)=e_1Ae_1\oplus\cdots\oplus e_rAe_r = (\End_DV_1)
  \oplus\cdots\oplus (\End_DV_r)
  \] 
and  $\dim_DV_i=\deg \End_D V_i$  for  $i=1,\dots,r$. 
Hence, all simple components of $C_A(L)$ have the same degree if and only if $\dim_DV_i=\frac1r\dim_DV$ for $i=1,\dots,r$. 
Assuming that $L\subseteq \Symm(\s)$, it is clear that for $i=1,\dots,r$ the involution $\s$ restricts to an involution of the same type on $e_iAe_i$ except possibly if $\s$ is orthogonal and $\car F=2$, and in that case the condition holds if and only if the restriction of $h$ on $V_i$ is non-alternating. This shows $(iv)$.
\end{proof}

\begin{cor}
  \label{cor:existsplitneat}
Let $r$ be a positive integer.
There exists a split neat $F$-subalgebra of $(A,\s)$ of degree $r$ if and only if $r$ divides $\coind A$ and $
\kap(A,\s)$.
\end{cor}
\begin{proof}
If $(A,\s)$ is split symplectic, then $\coind A=2\kap(A,\s)$, and in this case we set $d=\kap(A,\s)$.
In any other case $\coind A$ divides $\kap(A,\s)$, and we set $d=\coind A$.
In view of Corollary~\ref{C:neat-degrees} the degree of any neat $F$-subalgebra of $(A,\s)$ divides $\kap(A,\s)$.
On the other hand, the degree of a split neat $F$-subalgebra of $(A,\s)$ clearly divides $d$.

We claim that $(A,\s)$ contains a split neat $F$-subalgebra $L$ of
degree $d$. 
Once this is shown, assuming that $r$ divides $d$, we may choose an
$F$-subalgebra $K$ of $L$ with $[K:F]=r$ and such that $L$ is free as
a $K$-module, and obtain that $K$ is split and neat in
$(A,\s)$, by Lemma~\ref{lem:subneat}. 

A central simple algebra $A_0$ contains an $F$-subalgebra
$L_0$ isomorphic to $F^r$ and such that $A_0$ is free as a left
$L_0$-module if and only if $r$ divides $\coind A_0$. This together
with Proposition~\ref{P:inner-unitary-neat} 
shows the statement in the case where $(A,\s)$ is unitary of inner type.
Hence we may for the rest of the proof assume that $A$ is simple.

Assume that $(A,\s)$ is split. Then $d=\kap(A,\s)$ and it follows from Proposition~\ref{lem:capmaxdimsplit} together with Proposition~\ref{prop:maxneat} that $(A,\s)$ contains a split neat $F$-subalgebra $L$ with $[L:F]=\kap(A,\s)$.

For the rest of the proof we may in particular assume that $(A,\s)$
is not split symplectic, whereby $d=\coind A$. 
We identify $A$ with $\End_DV$ where $D$ is a non-commutative division $F$-algebra
and $V$ is a finite-dimensional right $D$-vector space. 
Then $d=\dim_DV$.

Since $(A,\s)$ is not split symplectic we may fix an $F$-involution $\tau$ on $D$ of the same type as $\s$.
Then $\s$ is adjoint to a 
hermitian form $h:V\times V\to D$ with respect to $\tau$. 
Since $(D,\tau)$ is in particular not split symplectic, we may diagonalise $h$.
In other words, we find a $D$-basis $(v_1,\dots,v_d)$ of $V$ which is orthogonal for $h$. 
Letting $a_i=h(v_i,v_i)$ for $i=1,\dots,d$,
we obtain that $a_1,\dots,a_d\in\Symm(D,\tau)$ and $$h\simeq \la a_1,\dots,a_d\ra\,.$$
Note that $h$ is alternating if and only if $\car F=2$ and $\s$ is symplectic, and in this case $a_1,\dots,a_d\in\Symd(D,\tau)$.
If $\car F=2$ and $\s$ is orthogonal, then at least one of $a_1,\dots,a_d$ is not contained in $\Symd(D,\tau)$.
However, if $\car F=2$ then for any $x\in\Symm(D,\tau)\setminus\Symd(D,\tau)$ and $a\in\Symd(D)\setminus\{0\}$
we have 
$$\la x,a\ra\simeq\la x+a,x^{-1}+a^{-1}\ra$$
and $x+a,x^{-1}+a^{-1}\in\Symm(D,\tau)\setminus\Symd(D,\tau)$.
Hence, if $\car F=2$ and $\s$ is orthogonal, then one can change the diagonalisation appropriately and assume that $a_1,\dots,a_d\in\Symm(D,\tau)\setminus\Symd(D,\tau)$.
Then the orthogonal basis $(v_1,\dots,v_d)$ yields an orthogonal decomposition of $(V,h)$ in $r$ subspaces of equal dimension, and the restriction of $h$ to any of these subspaces is only alternating in the case where $\car F=2$ and $\s$ is symplectic.
By Proposition~\ref{prop:neat-hermitian} the projections corresponding
to this decomposition generate a split neat $F$-subalgebra of $(A,\s)$ of degree~$d$. 
\end{proof}

\section{Neat quadratic subalgebras} 

\label{S:neat-quad}

Throughout this section let $(A,\sigma)$ be an $F$-algebra with
involution and let $K$ be a neat quadratic $F$-subalgebra of $(A,\sigma)$.
We shall prove that there exists a maximal neat subalgebra of $(A,\s)$ of the form $KL$ for a neat $F$-subalgebra $L$ of $(A,\s)$ which is $F$-linearly disjoint from $K$. This result will be crucial for our main results in the final section.

Set $C=C_A(K)$ and 
$$C'=C'_A(K)=\{x\in A\mid xk=\gamma(k)x\mbox{ for all }k\in K\}$$
where $\gamma$ denotes the nontrivial $F$-automorphism of $K$.
 
\begin{prop}
  \label{lem:neatquad1}
We have $A=C\oplus C'$ and
the $F$-vector spaces $C$ and $C'$ are stable under $\s$ and satisfy $$C\cap \Symd(\s)=\Symd(\s|_C)  \,\,\mbox{ and }\,\,C'\cap\Symm(\sigma) = C'\cap\Symd(\sigma)\,.$$ Moreover,
  $\dim_F C' =\dim_FC=
  {\textstyle\frac12}\dim_FA$ and $\dim_F\bigl(C'\cap\Symm(\sigma)\bigr)=\frac14\dim_FA$.
  \end{prop}

\begin{proof}
We fix $u\in K\setminus F$ with $u^2-u\in F$ and set
  $c=u^2-u$.
  Hence $\gamma(u)=1-u$ and we have $4c+1\neq0$ because the roots of the  polynomial $X^2-X-c$ are simple.
We obtain that $C'=\{x\in A\mid xu+ux=x\}$. 
Since $\s(u)=u$ we have $\s(C')=C'$.

Consider the $F$-linear map
\[ 
  \varphi\colon A\to A,\quad
  x\mapsto\textstyle{\frac1{4c+1}}\bigl((2c+1)x-ux-xu+2uxu\bigr).
\]
Computation shows that $\varphi(x)\in C$ and $x-\varphi(x)\in  C'$ for
any $x\in A$, 
and moreover $\varphi(x)=x$ for $x\in C$ and $\varphi(x)=0$ for $x\in
C'$.
Therefore $$A=C\oplus C'\,.$$ 
As $K$ is neat in $(A,\sigma)$ and $[K:F]=2$, it follows from Proposition~\ref{prop:neat} in the case where $A$ is simple and otherwise from Proposition~\ref{P:inner-unitary-neat} that $\dim_FC=\frac12\dim_FA$, whereby $\dim_F C'=\frac12\dim_FA$. 

Now, consider the $F$-linear map $f\colon C'\to C',x\mapsto x+\s(x)$ and set $W=\ker(f)$ and $U=f(C')$.
Hence $W=C'\cap \Skew(\s)$ and $\dim_FC'=\dim_FW+\dim_F U$. 
For $x\in C'\cap \Symm(\s)$ we have
$x=xu+ux=xu+\sigma(xu)\in U\subseteq C'\cap\Symd(\sigma)$, whereby $$U=C'\cap \Symm(\s)=C'\cap\Symd(\s)\,.$$
As $(1-2u)^2=1+4c\in\mg{F}$, multiplication from the left by $1-2u$ yields $F$-iso\-morphisms between $U$ and $W$. Hence $\dim_FU=\dim_FW=\frac12\dim_FC'$.

Clearly we have $\Symd(\s|_C)\subseteq C\cap\Symd(\s)$.
To show the converse inclusion, we may obviously assume that $\car F= 2$, whereby $u^2+u=c\in F$.
Consider $z\in C\cap \Symd(\s)$.
Let $x\in A$ be such that $z=\s(x)+x$.
As $u\in K\subseteq C$ we have $zu=uz$ and obtain that 
$\s(xu+ux)=(z+x)u+u(z+x)=ux+xu$
and thus 
$$z=\s(xu+ux+x)+xu+ux+x\,.$$
Since $u(ux+xu+x)=(u^2+u)x+uxu=x(u^2+u)+uxu=(ux+xu+x)u$ we further have  $ux+xu+x\in C_A(u)=C$ and conclude that $z\in \Symd(\s|_C)$.
\end{proof}
\begin{rem}
The definition of $\varphi$ in the proof comes from the  observation  that $\frac1{4c+1}(2c+1-u\otimes1-1\otimes
u + 2 u\otimes u)$ is the separability idempotent of $K$.
\end{rem}

Let us consider in more detail the case where $K$ is a
field. We then consider $C$ as a $K$-algebra and denote by $\s_C$ the $K$-involution on $C$ obtained by restricting~$\s$.

\begin{prop}\label{lem:PC3}
Assume that $K$ is a field.
  For $a\in
C'\cap \Symm(\sigma)$ we  have $$a^2\in\Syms(\sigma_C)\mbox{ and
}\chi_{A,a}(X)=\chi_{C,a^2}(X^2)\in F[X^2]\,.$$
\end{prop}

\begin{proof}
As $a\in
C'\cap \Symm(\sigma)$ we have that $a^2\in
C\cap\Symm(\sigma)$. 
If
$\sigma$ is symplectic then $\sigma_C$ is symplectic by Proposition~\ref{P:inv-sym-restrict-type} and for
$\ell\in C$ satisfying $\ell+\sigma(\ell)=1$ we obtain that $a^2=a\ell a+\sigma(a\ell a)$.
This shows that $a^2\in \Syms(\sigma_C)$. 

From Corollary~\ref{lem:PC2} we obtain the equality $$\Prd_{A,a}(X)=\Prd_{C,a^2}(X^2)\,.$$ 
We conclude that $\chi_{C,a^2}(X^2)=\chi_{A,a}(X)\in F[X^2]$.
\end{proof}

Back in the more general situation where $K$ is a neat quadratic $F$-algebra, but not necessarily a field, we conclude the following.

\begin{cor}\label{C:prd-anticom-square}
For $a\in C'\cap \Symm(\sigma)$ we have $\chi_{a}(X)\in F[X^2]$.
\end{cor}
\begin{proof}
If $K\simeq F\times F$ then the statement follows from Proposition~\ref{lem:PC}. Otherwise $K$ is a field, so that the statement follows from Proposition~\ref{lem:PC3}. 
\end{proof}

\begin{prop}\label{P:cap2-symd-anticom-space}
Assume that $\kap(A,\s)=2$.
Then $$\Syms(\s)=K\oplus(C'\cap\Symm(\s))$$
and this decomposition is orthogonal for the quadratic form $c_2:\Syms(\s)\to F$.
Furthermore 
 $c_2|_K:K\to F$ is the norm form of $K$ and  $c_2(x)=-x^2\in F$ for all $x\in C'\cap\Symm(\s)$.
\end{prop}
\begin{proof}
Set $V=\Syms(\s)$ and $W=C'\cap\Symm(\s)$.
It follows from Proposition~\ref{lem:neatquad1} and by comparing dimensions that $V=K\oplus W$.
Writing $\ovl{x}=c_1(x)-x$ for $x\in V$ defines an $F$-linear map $V\to V,x\mapsto \ovl{x}$. By Proposition~\ref{P:cap2-c2-regquad} we have $c_2(x)=x\ovl{x}$ for any $x\in V$.

For $x\in F$ we have $\chi_x=(X-x)^2$, whereby $c_1(x)=2x$ and $c_2(x)=x^2$.
For $x\in V\setminus F$ we have that $\chi_x$ is the minimal polynomial of $x$ over $F$.
From this we conclude that $\ovl{x}=\gamma(x)$ for $x\in K$ and that $c_2|_K$ is the norm form of $K$.

For $w\in W$ we have $w^2\in C\cap\Syms(\s)\cap C_A(w) =F$, whence $\chi_w=X^2-w^2$ and $w+\ovl{w}=c_1(w)=0$. For $v\in K$ and $w\in W$, using that  $vw=w\gamma(v)=w\ovl{v}$ and $w+\ovl{w}=0$
 we obtain that $$c_2(v+w)-c_2(v)-c_2(w)=v\ovl{w}+w\ovl{v}=v(w+\ovl{w})=0\,.$$
This shows that $K$ is orthogonal to $W$ with respect to $c_2$ and that $\ovl{x}=-x$ and $c_2(x)=-x^2$ for $x\in W$.
\end{proof}

\begin{cor}\label{T:cap2}
If $\kap(A,\s)=2$, then $K$ is contained in a $\s$-stable quaternion $F$-subalgebra of $A$.
\end{cor}
\begin{proof}
Let $W=C'\cap \Symm(\s)$.
It follows from  Proposition~\ref{P:cap2-c2-regquad} and Proposition~\ref{P:cap2-symd-anticom-space} that the quadratic form $c_2|_W:W\to F,x\mapsto -x^2$ is nondegenerate. As $W\neq\{0\}$ it follows that there exists $x\in W$ with $x^2\in\mg{F}$.
Since $Kx\subseteq C'$ we conclude that $K\oplus Kx$ is a $\s$-stable quaternion $F$-subalgebra.
\end{proof}

For $a\in C'\cap \Symm(\s)$, by Corollary~\ref{C:prd-anticom-square} there is a unique polynomial $f\in F[X]$ with $\chi_a(X)=f(X^2)$, and we call the element $a$ \emph{square separable} if $f$ is separable.
 
\begin{prop}\label{cor:Zariski-reductiontosplit}
Assume that $F$ is algebraically closed. In $A$ the set $$\{a\in {C'\cap\Symm(\s)}\mid a\mbox{ is square separable} \}\cap \mg{A}$$ is open in ${C'\cap\Symm(\s)}$ with respect to the Zariski topology.
\end{prop}
\begin{proof}
A polynomial in $F[X]$ is inseparable if and only if its discriminant vanishes.
Hence, for $a\in C'\cap\Symm(\s)$ being square separable is characterised by the nonvanishing of a polynomial in the coefficients of $\chi_{a}(X)$, which in turn are polynomials in the coefficients of $a$ with respect to any fixed $F$-basis of $C'\cap \Symm(\s)$.
Therefore in $C'\cap \Symm(\s)$ the square separable elements form an open subset with respect to the Zariski topology.
On the other hand, in $A$ the invertible elements are characterised by the nonvanishing of the reduced norm, whereby $\mg{A}$ is open in $A$. 
The statement follows from these two observations by basic topology.
\end{proof}

\begin{prop}\label{lem:splitnonempty}
If $|F|>\kap(A,\s)$
then $C'\cap\Symm(\s)\cap\mg{A}$ contains a square separable element.
\end{prop}
\begin{proof}
Suppose first that $F$ is finite or algebraically closed.
Then $(A,\s)$ is split.
Set $r=\frac12\kap(A,\s)$. By the hypothesis we have that $|{F^{\times 2}}|\geq r$.
If $K$ is a field then $(C,\s_C)$ is a split $K$-algebra with involution of the same type as $(A,\s)$ and with $\kap(C,\s_C)=r$, so that by Corollary~\ref{cor:existsplitneat} there exists a split neat $K$-subalgebra $L$ of $(C,\s_C)$ with $[L:K]=r$.
If $K\simeq F\times F$ then by Theorem~\ref{thm:neat-embed-maxetsym}, $K$ is contained in a split neat $F$-subalgebra $L$ of $(A,\s)$ with $[L:F]=2r$.
In either of these two cases we have that $L\simeq K^r$ as $F$-algebras.
Let $e_1,\dots,e_r\in L$ be the corresponding idempotents in $L$ satisfying $K\simeq Ke_i$ for $i=1,\dots,r$.
If $K\simeq F\times F$, then we fix a primitive idempotent $f\in K$ and obtain that
$e_1f,\dots,e_rf,e_1(1-f),\dots,e_r(1-f)$ are the primitive idempotents in $L$.
For $i=1,\dots,r$, we set $A_i=e_iAe_i$ and $\s_i=\s|_{A_i}$ and obtain by identifying $F$ with $Fe_i\subseteq A_i$ that $(A_i,\s_i)$ is a split $F$-algebra with involution of the same type as $(A,\s)$ and with $\kap(A_i,\s_i)=\frac{1}r\kap(A,\s)=2$.

Consider $i\in\{1,\dots,r\}$.
By Corollary~\ref{T:cap2} the neat quadratic $F$-subalgebra $Ke_i$ of $A_i$ is contained in a $\s_i$-stable quaternion $F$-subalgebra $Q_i$ of $A_i$.
By the assumption on $F$, $Q_i$ is split, and since $Ke_i\subseteq \Symm(\s_i)$ we obtain by Theorem~\ref{thm:capmaxdim} that $\s_i|_{Q_i}$ is orthogonal. It follows that there exists $g_i\in Q_i\cap\Symm(\s_i)$ with $g_i^2\in F e_i$ and such that 
$\Int_{Q_i}(g_i)$ restricts to the nontrivial $F$-automorphism on $Ke_i$. 
Note that $g_i$ is determined by this property up to a multiple in $\mg{F}$.
Moreover, since $|{F^{\times 2}}|\geq r$, 
 we may choose $g_1,\dots,g_r$ in such way that 
$g_i^2=c_ie_i$ for $i\in\{1,\dots,r\}$ where $c_1,\dots,c_r\in\mg{F}$ are pairwise distinct.
For $g=g_1+\dots+g_r\in C'\cap\Symm(\s)$ it follows that $\Prd_{A,g}=\prod_{i=1}^r(X^2-c_i)$, whereby $g$ is invertible and square separable.

We turn to the general case, where we may assume that $F$ is an infinite field.
We choose an algebraic closure  $\overline F$ of $F$ and consider the $\overline F$-algebra with involution $(\overline A,\overline \s)$ naturally obtained from $(A,\s)$ by letting $\overline A=A\otimes_F\overline F$ and $\overline \s=\s\otimes \id_{\overline F}$.
For any $F$-subspace $W$ of $A$ we write $\overline W=W\otimes_F \overline F$ and note that $W$ is dense in $\overline W$ for the Zariski topology.
By the above and by Proposition~\ref{cor:Zariski-reductiontosplit} the elements of  $\overline{C'}\cap \Symm({\overline \s})$ that are square separable and invertible in $\overline A$ form a nonempty  Zariski-open subset of $\overline{C'}\cap\Symm(\overline\sigma)$. 
As $\overline{C'}\cap\Symm(\overline\sigma)=\overline{C'\cap\Symm(\s)}$, we obtain that there exists a square separable invertible element in $C'\cap\Symm(\s)$.
\end{proof}

\begin{prop}\label{prop:sepelement}
Let $a\in C'\cap\Symm(\s)\cap\mg{A}$ be square separable.
Then $F[a^2]$ and $K[a^2]$ are neat $F$-subalgebras of $(A,\s)$ such that $F[a^2]$ is $F$-linearly disjoint from $K$ and $[K[a^2]:F]=[F[a]:F]=\kap(A,\s)$.
\end{prop}
\begin{proof}
By the hypothesis $\chi_a(X)=f(X^2)$ for a separable polynomial $f\in F[X]$.
In particular $\deg f=\frac12\deg \chi_a=\frac12\kap(A,\s)$.
As $f$ is separable and $f(a^2)=0$, the $F$-algebra $F[a^2]$ is \'etale and $[F[a^2]:F]=\deg f=\frac12 \kap(A,\s)$.
Since $a\in C'$ we have that $a^2\in C$, whereby $F[a^2]\subseteq C$ and $aF[a^2]\subseteq C'$, which shows that $F[a]=F[a^2]\oplus aF[a^2]$. Since $a\in\mg{A}$ it follows that $$[F[a]:F]=2\cdot [F[a^2]:F]=\kap(A,\s)\,.$$
As $a^2\in C$ the $F$-algebra $K[a^2]$ is commutative. 
Conjugation by $a$ restricts to a nontrivial $F$-automorphism of order two on $K[a^2]$ which fixes $F[a^2]$. It follows that $K$ is $F$-linearly disjoint from $F[a^2]$ and that $[K[a^2]:F[a^2]]=2$.
Hence, $[K[a^2]:F]=\kap(A,\s)$ and $K[a^2]$ is neat in $(A,\s)$, by Proposition~\ref{prop:maxneat}.
Since $K[a^2]$ is free as an $F[a^2]$-module, it follows by Lemma~\ref{lem:subneat} that $F[a^2]$ is neat in $(A,\s)$.
\end{proof}

\begin{thm}\label{thm:neatquad}
There exists a neat subalgebra $L$ of $(A,\sigma)$ contained in $C_A(K)$, 
   $F$-linearly disjoint from $K$ and such that $KL$ is a neat subalgebra of $(A,\sigma)$  with $[KL:F]=\kap(A,\sigma)$.
\end{thm}
\begin{proof}
Assume that $(A,\s)$ is split.
If $K\simeq F\times F$, then the statement follows immediately from Theorem~\ref{thm:neat-embed-maxetsym}.
Suppose now that $K$ is a field.
Set $C=C_A(K)$ and $\s_C=\s|_C$.
By Proposition~\ref{P:inv-sym-restrict-type} we obtain that $(C,\s_C)$ is a split $K$-algebra with involution such that $\kap(C,\s_C)=\frac12\kap(A,\s)$. It follows by Proposition~\ref{lem:capmaxdimsplit} that $\Syms(\s_C)$ contains a split \'etale $K$-algebra $M$ of $A$ with $[M:K]=\kap(C,\s_C)$. By Proposition~\ref{prop:maxneat} we have that $M$ is neat in $(A,\s)$.
 Let $L$ be the $F$-subalgebra of $M$ generated by the idempotent elements in $M$. Then $L$ is $F$-linearly disjoint from $K$ and $KL=M$, thus $[KL:F]=[M:K]\cdot [K:F]=\kap(A,\s)$.
Furthermore, $M$ is free as an $L$-module, whence $L$ is neat in $(A,\s)$ by Lemma~\ref{lem:subneat}. 

Hence the statement holds when $A$ is split.
In particular it holds when $F$ is finite.
Assume now that $F$ is infinite.
By Proposition~\ref{lem:splitnonempty} there exists an element $a\in C'\cap\Symm(\s)\cap\mg{A}$ which is square separable.
Then Proposition~\ref{prop:sepelement} shows that $L=F[a^2]$ has the desired property.
\end{proof}

\begin{rem}
If $\car F\neq 2$ then instead of the set in Proposition~\ref{cor:Zariski-reductiontosplit} one may consider the set $\{a\in C'\cap\Symm(\s)\cap\mg{A}\mid \chi_{a}\mbox{ separable}\}$.
To see that this set is Zariski-open in $C'\cap\Symm(\s)$ when $F$ is algebraically closed is easier, as it does not involve Proposition~\ref{lem:PC}, Corollary~\ref{lem:PC2} and Proposition~\ref{lem:PC3}.
Note however that this set is empty if $\car F=2$.
\end{rem}

\section{Capacity four} 
\label{S:cap4}

In this section we consider in more detail algebras with involution of capacity four and show the existence of  biquadratic neat subalgebras (Theorem~\ref{T:neat-quad-sym-exist}). 
We shall in particular be interested in the case of symplectic involutions on algebras of degree eight.
In this case we will conclude the existence of a triquadratic \'etale extension of the centre which is stable under the involution (Theorem~\ref{T:symp-deg8-triquad}).
In particular, we obtain a new proof to Rowen's Theorem stating that every degree eight algebra of exponent two contains a triquadratic \'etale subalgebra (Corollary~\ref{C:AR}).

We need the following two preparatory results, which are  well-known.

\begin{prop}
  \label{L:Merk-triv}
Assume that $A$ is a central simple $F$-algebra. There exist $r\in\nat$ and a sequence of separable quadratic field extensions $(F_i/F_{i-1})_{i=1}^r$ with $F_0=F$ such that $\ind A_{F_r}$ is odd. 
\end{prop}

\begin{proof}
Primary decomposition (cf.~\cite[Proposition~4.5.16]{GS06}) yields that $A\simeq B\otimes C$ for two central simple $F$-algebras $B$ and $C$ such that $\ind B$ is odd and $\ind  C=2^m$ for some $m\geq 1$.
Then $C$ represents an element of order dividing $2^m$ in the Brauer group of $F$.
By \cite[Theorem]{Bec16} there exist
 $r\in\nat$ and a sequence of separable quadratic field extensions  $(F_i/F_{i-1})_{i=1}^r$ with $F_0=F$ and such that $C_{F_r}$ is split. 
(Alternatively, this can be derived from Merkurjev's Theorem \cite[Theorem~1.5.8]{GS06}.)
It follows that $\ind A_{F_r}$ divides $\ind B$. 
\end{proof}

\begin{prop}[Springer]
  \label{P:cubic-Springer}
Any cubic form over $F$ which has a nontrivial zero over a quadratic field extension of $F$ also has a nontrivial zero over $F$.
\end{prop}

\begin{proof}
  Consider a cubic form $f$ in $n$ variables over $F$.  
We suppose that $f$ has a nontrivial zero in $F[X]/(p)$ for some irreducible quadratic polynomial $p\in F[X]$.
  Hence there exist $b_1$, $c_1$, \ldots, $b_n$, $c_n\in F$, not all
  zero, and $h\in F[X]$ such that
  \[
  f(b_1+c_1X,\dots,b_n+c_nX)= p(X)\cdot h(X)\,.
  \]

  Suppose first that $h\in F$.  Comparing coefficients in degree $3$
  we obtain that $f(c_1,\dots,c_n)=0$.  Moreover, if
  $c_1=\dots=c_n=0$ then we obtain further that
  $f(b_1,\dots,b_n)=0$.  As $b_1$, $c_1$, \ldots, $b_n$, $c_n\in F$ are not
  all zero, it follows that $f$ has a nontrivial zero in $F$.

Suppose now that $h\notin F$.  
As
  $\deg(f(b_1+c_1X,\dots,b_n+c_nX))\leq \deg(f)=3$ and
  $\deg(p)=[K:F]=2$, we conclude that $\deg(h)= 1$.  Hence there exists $a\in F$ such that $h(a)=0$.  
Then
  $f(b_1+c_1a,\dots,b_n+c_na)=0$.  Moreover, if $b_i+c_ia=0$ for
  $i=1$, \ldots, $n$, then
  $0=f(c_1(X-a),\dots,c_n(X-a))=(X-a)^3f(c_1,\dots,c_n)$ and
  thus $f(c_1,\dots,c_n)=0$.  Hence $f$ has a nontrivial zero in $F$.
\end{proof}
 
Now let $(A,\sigma)$ be an $F$-algebra with involution. 
Recall that for  $d=\kap(A,\sigma)$ and $a\in \Syms(\s)$ we have $$\chi_a=X^d-c_1(a)X^{d-1}+c_2(a)X^{d-2}+\dots+(-1)^dc_d(a)\in F[X]\,.$$

\begin{lem}
  \label{lem:c1c3-vanish}
  Assume that $\kap(A,\sigma)$ is a multiple of $4$.
  There exists an element $a\in \Syms(\sigma)\setminus F$ such that
  $c_1(a)=c_3(a)=0$. 
\end{lem}

\begin{proof}
  Let $d=\kap(A,\sigma)$.
  We first consider the situation where $\coind(A)$ is even.  
 Then by Corollary~\ref{cor:existsplitneat} there exists a split neat $F$-subalgebra $L$ of $(A,\s)$ with $[L:F]=2$.  
If $\car F \neq 2$ we choose an element $a \in L\setminus F$ with $a^2=1$ and obtain that $\chi_a(X)=(X^2-1)^{d/2}$.  If $\car F =2$ then we choose $a\in L\setminus F$ with $a^2=a$ and obtain that
  $\chi_a(X)=X^{d}+X^{d/2}$.  In either case we have that $a\in \Syms(\sigma)\setminus F$ and $\chi_a(X)\in F[X^2]$. Hence $a$ has the desired
  properties.

In the general case, by Proposition~\ref{L:Merk-triv} there exists $r\in\nat$ and a sequence of quadratic field extensions
  $(F_i/F_{i-1})_{i=1}^r$ with $F_0=F$ such that $\ind A_{F_r}$ is odd.
In particular, $\coind A_{F_r}$ is even.

Let $W=\ker(c_1)\subseteq \Syms(\sigma)$.  If $\car F \neq 2$, then $F\cap
  W=0$, and we  consider the cubic form $f=c_3$ on $W$.  If $\car F =2$, then we
  have $c_3(x+a)=c_3(x)$ for every $x\in W$ and every $a\in F$, for $\chi_{x+a}(X)=\chi_x(X-a)$. 
In this 
  case we consider the cubic form 
  $f\colon W/F\to F, x+F\mapsto c_3(x)$.  In each case the validity of the statement is equivalent
  to the existence of a nontrivial zero of the cubic form~$f$.
By the special case considered above, $f$ has a nontrivial zero in $F_r$. Since $f$ is a cubic form and $F_i/F_{i-1}$ is a quadratic extension for $i=1$, \dots, $r$, we conclude by Proposition~\ref{P:cubic-Springer} that $f$ has a nontrivial zero over $F$.
\end{proof}

An \'etale $F$-algebra $K$ is called \emph{biquadratic} (resp.~\emph{triquadratic}) if it is isomorphic to the tensor product of two (resp.~three) quadratic \'etale $F$-algebras.

The following result extends \cite[Corollary~6.2 and Theorem 9.1~(1)]{GPT09}.

\begin{thm}\label{T:neat-quad-sym-exist}
  Let $(A,\sigma)$ be an $F$-algebra with involution with
  $\kap(A,\sigma)=4$. Then $(A,\sigma)$ contains a neat biquadratic $F$-subalgebra.
\end{thm}

\begin{proof}
  By Theorem~\ref{thm:neatquad} it suffices to show that $(A,\s)$ contains a neat quadratic $F$-subalgebra $K$.  If
  $A$ has zero-divisors, then we may conclude this 
  by applying Corollary~\ref{cor:existsplitneat} with $r=2$.
Hence we assume that $A$ is a division $F$-algebra.

By Lemma~\ref{lem:c1c3-vanish} there exists an element $a\in
  \Syms(\sigma)\setminus F$ with $c_1(a)=c_3(a)=0$, whereby $\chi_a(X)=X^4+c_2(a)X^2+c_4(a)\in F[X^2]$. 
In particular $[F(a^2):F]\leq 2$.  
We set $E=F[a]$ if $a^2\in F$ and $E=F[a^2]$ otherwise.  Then $E$ is a quadratic field extension of $F$ contained in $\Symm(\sigma)$.

  If $E$ is separable over $F$, then $E$ is a neat subalgebra of
  $(A,\sigma)$ and we may take $K=E$.
  Suppose now that the quadratic extension $E$ is inseparable. In
  particular $\car F =2$.  We consider $C=C_A(E)$ and write $\sigma_C$
  for the restriction of $\sigma$ to $C$. If we can find 
  $y\in \Symm(\sigma_C)\setminus E$ such that $y^2+y\in E$, we obtain
  for $u=y^2$ that
  $u\in\Symm(\sigma_C)\setminus E\subseteq \Symm(\sigma)\setminus F$
  and $u^2+u=(y^2+y)^2\in F$, whereby $F[u]$ is a separable quadratic
  extension of $F$ contained in $\Symm(\sigma)$, so that we may take
  $K=F[u]$.  It therefore suffices to 
  show the existence of such an element $y$.

  Note that $(C,\sigma_C)$ is an $E$-algebra with involution and
  $\deg(C)=\frac{1}{2}\deg(A)$.  If $\kap(C,\sigma_C)=2$, then the
  existence of $y\in \Symm(\sigma_C)\setminus E$ with $y^2+y\in E$
  follows by Theorem~\ref{thm:capmaxdim}.  The only possibility to
  have $\kap(C,\sigma_C)\neq 2$ is that $(A,\sigma)$ is symplectic of
  degree $8$ and $(C,\sigma_C)$ is orthogonal of degree $4$.

  In particular, the statement holds in the case where $\sigma$ is
  orthogonal. Applying this to $(C,\s_C)$ when
   $\sigma_C$ is orthogonal and $\deg(C)=4$, we obtain  a separable quadratic extension of $E$ inside $\Symm(\sigma_C)$ and thus an element $y\in \Symm(\sigma_C)\setminus E$ with
  $y^2+y\in E$, as desired. 
\end{proof}

The proof of our next result uses a corestriction argument on central simple algebras.
Consider a separable quadratic field extension $K/F$ and  a central simple $K$-algebra $B$.
We refer to \cite[\S3.B]{BOI} for the definition and the basic properties of the central simple $F$-algebra $\Cor_{K/F}(B)$, the corestriction (or norm) of $B$ from $K$~to~$F$  (which is denoted $N_{K/F}(B)$ in \cite{BOI}).

\begin{lem}\label{L:AlbertDraxlTransfer}
Assume that $B$ is a $K$-quaternion algebra.
  \label{cor:Albert}
Then $B$ contains a quadratic \'etale $F$-algebra linearly disjoint from $K$
 if and only if $\Cor_{K/F}(B)$ is not a division algebra.
\end{lem}
\begin{proof}
If $\car F\neq 2$, a proof is given in \cite[(16.28)]{BOI}.
We refer to \cite{BGBT-AlbertDraxlTransfer} for a proof in arbitrary characteristic.
\end{proof}

The following result extends \cite[Lemma~6.1]{GPT09}.

\begin{thm}\label{T:symp-deg8-triquad}
Assume that $(A,\s)$ is symplectic of degree $8$.
Then $A$ contains a $\s$-stable triquadratic \'etale $F$-subalgebra.
Moreover, any neat biquadratic $F$-sub\-algebra of $(A,\s)$ is contained in a $\s$-stable triquadratic \'etale $F$-subalgebra of~$A$.
\end{thm}
\begin{proof}
In view of Theorem~\ref{T:neat-quad-sym-exist} it suffices to prove the second part of the statement.
Thus let $L$ be a neat biquadratic $F$-subalgebra of $(A,\s)$.

Assume first that $L$ is split.
Let $e_1,\dots,e_4$ be the primitive idempotents of $L$.
For $i=1,\dots,4$ by identifying $F$ with $Fe_i$ we obtain that
$e_iAe_i$ is a quaternion $F$-algebra Brauer equivalent to $A$ and $\s|_{e_iAe_i}$ is its canonical involution.
As $\deg A=8$ and $A$ contains a split biquadratic \'etale $F$-subalgebra, we have $\ind A\leq 2$.
If $\ind A=2$ then we fix $a\in F$ with $4a\neq -1$ such that $F[X]/(X^2-X-a)$ is a splitting field of $A$, otherwise we set $a=0$.
In either case, we obtain for $i=1,\dots,4$ an element $f_i\in e_iAe_i\setminus Fe_i$ with $f_i^2=f_i+ae_i$.
Then $f=f_1+f_2+f_3+f_4$ is such that $f+\s(f)=1$ and $f^2=f+a$.
Hence $L[f]$ is a $\s$-stable triquadratic $F$-subalgebra of $A$.

Assume now that $L$ is not split. 
Then $L$ contains a quadratic field extension $K$ of $F$. 
With the notation of Section~\ref{S:neat-quad} we obtain an 
$F$-algebra with involution $(C,\s_C)$.
By Proposition~\ref{P:inv-sym-restrict-type} the involution $\s_C$ is symplectic and $\kap(C,\s_C)=2$.
Since $[L:K]=2=\kap(C,\s_C)$, by Proposition~\ref{prop:maxneat} the \'etale $K$-algebra $L$ is neat in $(C,\s_C)$.
By Corollary~\ref{T:cap2} it follows that $L$ is contained in a $\s_C$-stable quaternion $K$-subalgebra $Q$ of $C$.
We set $Q'=C_C(Q)$ and observe that $Q'$ is a $\s_C$-stable quaternion $K$-subalgebra of $C$.
We set $\s_Q=\s|_Q$ and $\s_{Q'}=\s|_{Q'}$ and obtain that $(Q,\s_Q)$ and $(Q',\s_{Q'})$ are quaternion $K$-algebras with involution such that
$$(C,\s_C)\simeq (Q,\s_Q)\otimes (Q',\s_{Q'})\,.$$
Since $L\subseteq\Symm(\s_Q)$ it follows from Theorem~\ref{thm:capmaxdim} that $\s_Q$ is orthogonal.
As $\s_C$ is symplectic, it follows by \cite[(2.23)]{BOI} that $\s_{Q'}$ is symplectic.
Hence $\s_{Q'}$ is the canonical involution of $Q'$.

The central simple $K$-algebra $C$ is Brauer equivalent to
$A_K$. Since $A$ carries an $F$-linear involution, $A\otimes_FA$ is
split.
This implies that $\Cor_{K/F}(C)$ is split, hence
$\Cor_{K/F}(Q)\simeq\Cor_{K/F}(Q')$. Since $L$ is biquadratic, we have
$L\simeq K\otimes_FM$ for some quadratic \'etale $F$-algebra $M$. As
$M\subseteq Q$, Lemma~\ref{L:AlbertDraxlTransfer} shows that $\Cor_{K/F}(Q)$ is not a
division algebra. Therefore $\Cor_{K/F}(Q')$ is not a division
algebra. Hence, by Lemma~\ref{L:AlbertDraxlTransfer} there exists  a quadratic \'etale
$F$-algebra $K'\subseteq Q'$ linearly disjoint from $K$.
Note that $K'$ is $\s$-stable, for $\s|_{Q'}$ is the canonical
involution of $Q'$. 
Note further that $K'\subseteq Q'= C_C(Q)\subseteq C_A(L)$.
Hence $LK'$ is a $\s$-stable triquadratic \'etale $F$-subalgebra of $A$.
\end{proof}

It is known that every central division algebra of exponent two and
degree at most eight has a maximal subfield that is a separable
multiquadratic extension of the centre. This was shown by Albert
\cite[Chapter~XI, Theorem~9]{SofA} for degree four and by Rowen
\cite[Theorem~1]{Rowen84} for degree eight.  We obtain a new proof of
this statement.

\begin{cor}[Albert, Rowen]
  \label{C:AR}
  Let $A$ be a central simple $F$-algebra such that $\deg(A)$ divides
  $8$ and $A\otimes_FA$ is split. Then $A$ contains a maximal
  commutative subalgebra that is an \'etale multiquadratic
  $F$-algebra. 
\end{cor}

\begin{proof}
Let $n\in\nat$ be such that $\deg(A)=2^n$.  If $n\leq 1$ then the statement is obvious. If $n=2$ then we choose an orthogonal involution $\sigma$ on $A$ and conclude by Theorem~\ref{T:neat-quad-sym-exist}.
If $n=3$ then we choose a symplectic involution $\s$ on $A$ and apply Theorem~\ref{T:symp-deg8-triquad}.
\end{proof}

\subsubsection*{Acknowledgments}
It is a pleasure to thank Holger Petersson for his comments on Jordan
algebraic aspects, Kader Bing\"ol for a critical reading of a preliminary version and
 the referee for pointing out a gap (and
  suggesting a patch) in a first version of
  Corollary~\ref{cor:existsplitneat}. 

This work was supported by the FWO Odysseus Programme (project
\emph{Explicit Methods in Quadratic Form Theory}), funded by the Fonds
Wetenschappelijk Onderzoek -- Vlaanderen. The third author
acknowledges support from the Fonds de la Recherche Scientifique--FNRS
under grants n$^\circ$~J.0014.15 and J.0149.17. Work on this paper was
initiated in 2011 while the first and the third author where,
respectively, Fellow and Senior Fellow of the Zukunftskolleg, whose
hospitality is gratefully acknowledged.  

\bibliographystyle{plain}

\end{document}